\documentclass[11pt]{article}
\usepackage{amsmath,amssymb,amsthm}
\usepackage{mathtools}
\usepackage{mathdots}
\usepackage{hyperref,url}
\hypersetup{
	colorlinks=true,
  linkcolor=black,          
  citecolor=black,         
  filecolor=black,      
  urlcolor=black           
}
\usepackage{graphicx}

\newcommand{\bigprime}[1]{%
  \raisebox{-0.15ex}{\scalebox{1.45}{$#1$}}%
}

\newcommand{\NN}{\mathbb{N}}
\newcommand{\ZZ}{\mathbb{Z}}

\newcommand{\RR}{\mathbb{R}}

\newcommand{\seqnum}[1]{\href{https://oeis.org/#1}{\rm \underline{#1}}}

\newcommand{\digitlen}[1]{\ell(#1)}

\title{Further proofs of conjectures from the OEIS}

\author{ Sela Fried\\ Department of Computer Science, Israel Academic College\\ Ramat Gan, Israel\\ \texttt{friedsela@gmail.com} }

\date{}

\theoremstyle{plain}
\newtheorem{theorem}{Theorem}
\newtheorem{lemma}[theorem]{Lemma}
\newtheorem{proposition}[theorem]{Proposition}
\newtheorem{corollary}[theorem]{Corollary}

\theoremstyle{definition}

\newtheorem{example}[theorem]{Example}

\theoremstyle{remark}
\newtheorem{remark}[theorem]{Remark}

\begin{document}
\maketitle

\begin{abstract}
This is the fourth work in a series devoted to proving conjectures recorded in the On-Line Encyclopedia of Integer Sequences (OEIS). The problems considered here concern elementary and multiplicative number theory, Fibonacci numbers, decimal concatenation, Diophantine and Pell equations, binary representations and bitwise operations, lattice paths, parity patterns, recurrences, and formal power series. Several of the results give complete characterizations of the relevant sequences; others establish exact identities, recurrences, generating functions, asymptotic estimates, integrality properties, or nonoccurrence results. The proofs use combinatorial bijections, congruences, M\"obius inversion, valuations, Fibonacci identities, Pell-type arguments, Lucas' theorem, Riordan arrays, Lagrange inversion, and generating-function methods.
\end{abstract}

\section{Introduction}

The On-Line Encyclopedia of Integer Sequences (OEIS) \cite{Sl} is a continually expanding repository of integer sequences. In addition to numerical data, its entries include references, formulas, links to related sequences, and conjectures or open questions contributed by users.

Continuing the work developed in \cite{Fr1, Fr2, Fr3}, the present paper contains results concerning more than thirty OEIS entries. These results include proofs of conjectured formulas and recurrences, characterizations of selected sequences, answers to open questions, and several strengthenings of previously stated conjectures. The manuscript is organized into two main parts: number-theoretic problems and problems involving generating functions.

The number-theoretic part addresses counting, divisibility, asymptotics, and Diophantine questions. Among other results, we give a counting interpretation of Euler’s totient function for \seqnum{A000010}, a primality characterization for \seqnum{A000680}, and an evaluation of a Fibonacci sum from \seqnum{A007598}. We also prove bounds for \seqnum{A078111}, derive a closed form for \seqnum{A111386}, prove that \seqnum{A116098} and \seqnum{A116129} coincide, and resolve problems involving powers, bitwise operations, Pell equations, and quadratic forms.

The second part concerns formal power series and generating functions. For example, we prove a theta-function identity for \seqnum{A000122}, establish a Riordan-array interpretation related to \seqnum{A108080}, prove integrality for \seqnum{A158110}, and derive generating functions and recurrences for restricted-word and inverse-series sequences. 

The proofs use bijections, congruences, M\"obius inversion, valuations, Fibonacci identities, Pell equations, coefficient extraction, Lucas' theorem, Riordan arrays, Lagrange inversion, and lattice-path arguments. In several cases, a stronger structural result yields the conjectured identity as an immediate consequence.

We begin by introducing some notation. We let $\NN$ denote the set of natural numbers $\{1,2,\ldots\}$ and set $\NN_0=\NN\cup\{0\}$. For $n\in\NN$ let $[n]=\{1,\ldots,n\}$ and set $[0]=\emptyset$. For a set $A$ we let $\# A$ denote the number of elements in $A$. For $n\in\NN$ we write $\digitlen{n}$ for the number of decimal digits of $n$. For a condition $c$, we write $\mathbf{1}_c$ for the indicator function for $c$, i.e., $\mathbf{1}_c=1$ if $c$ holds and $\mathbf{1}_c=0$ otherwise. For $n\in\NN$ let $\varphi(n)$ denote Euler's totient function. For $n\in\NN_0$, let $F_n$ denote the $n$th Fibonacci number. We extend $F_n$ to the negative integers via $F_{-n} = (-1)^{n+1}F_n$. For $n\in\NN_0$ let $L_n$ denote the $n$th Lucas number, and let $T_n=\frac{n(n+1)}{2}$ denote the $n$th triangular number. If a function $f(x)$ has the formal power series representation $f(x)=\sum_{n=0}^\infty a_n x^n$ and $k\in\NN_0$, we write $[x^k]f(x)$ for $a_k$. For $u,v\in\NN_0$ write $u\land v$ for their bitwise AND, $u\lor v$ for their bitwise OR, and $u\oplus v$ for their bitwise XOR, where the bitwise operations $\land$, $\lor$, and $\oplus$ are performed by writing $u$ and $v$ in binary and applying the corresponding Boolean operation to each pair of binary digits separately, and finally converting the resulting binary number back to decimal notation.

\section{Number theory}

\subsection{\texorpdfstring{\seqnum{A000010}}{}}

The statement of the following theorem was conjectured by Irwin in \seqnum{A000010}.

\begin{theorem}
Let $n\in\NN$ with $n\geq 2$ and let 
\[
\mathcal{S}(n) = \{(a,b,c)\in[n]^3 \;:\; n c - a b = 1\}.\] Then $\#\mathcal{S}(n)=\varphi(n)$.
\end{theorem}

\begin{proof}
Let
\[
\mathcal{U}(n) = \{a\in [n] \;:\; \gcd(a,n)=1\}.
\]
Since $\#\mathcal{U}(n)=\varphi(n)$, it suffices to show that the map $\Theta\colon \mathcal{S}(n)\to\mathcal{U}(n)$ defined by $\Theta(a,b,c)= a$, is a bijection. 
\begin{enumerate}
\item $\Theta$ is well-defined. Let  $(a,b,c)\in\mathcal{S}(n)$ and let $d = \gcd(a,n)$. Thus, $d\mid a$ and $d\mid n$ and therefore $d$ divides $nc-ab=1$. Hence, $d=1$ and therefore
$a\in\mathcal{U}(n)$. 
\item $\Theta$ is injective. Let $a\in\mathcal{U}(n)$ and assume there are $b,b',c,c'\in[n]$ such that
$(a,b,c),(a,b',c')\in\mathcal{S}(n)$.
Then $nc - ab =1= nc' - ab'$.
Thus, $n(c-c') = a(b-b')$. In particular, 
$n\mid a(b-b')$. Since $\gcd(a,n)=1$, necessarily
$n\mid(b-b')$. But $b,b' \in[n]$. Thus, $b-b'=0$, i.e., $b = b'$. From this it follows that $n(c-c')=0$ and therefore $c = c'$.

\item $\Theta$ is surjective. Let $a\in\mathcal{U}(n)$. Since $\gcd(a,n)=1$, by B\'{e}zout's identity there are $u,v\in\ZZ$ such that $ua + vn = 1$. Multiplying by $-1$ and
reducing modulo $n$, we obtain $(-u)a \equiv -1 \pmod{n}$. Let $b\in[n]$ such that $b\equiv -u \pmod n$. Then $ab \equiv -1 \pmod{n}$. Thus, there exists $c\in\ZZ$ such that $ab = nc - 1$ and therefore $nc-ab=1$. It remains to show that $c=\frac{ab+1}{n}\in[n]$. Since $a,b\in[n]$, we have $\frac{2}{n} \leq c \leq n+\frac{1}{n}$. 
Since $n\geq 2$ and $c$ is an integer, necessarily $c\in[n]$.\qedhere
\end{enumerate}
\end{proof}

\subsection{ \texorpdfstring{\seqnum{A000680}}{}}

The statement of Theorem \ref{tttte} below was conjectured by Schulte in \seqnum{A000680}. We shall need
a result known by the name Euclid's lemma (e.g., \cite[Theorem 2.5]{B}): Let $a,b,c\in\ZZ$ with $\gcd(a,b)=1$. If $a\mid bc$, then $a\mid c$.    

We shall also need a result called Fermat's little theorem (e.g., \cite[Theorem 5.1]{B}): Let $p$ be a prime number and let $a\in\ZZ$ with $p\nmid a$. Then $a^{p-1}\equiv1\pmod p$.

Finally, let us recall Wilson's theorem (e.g., \cite[Theorem 5.4]{B}): Let $p$ be a prime number. Then $(p-1)!\equiv-1\pmod p$.   

\begin{theorem}\label{tttte}
Let $n\in\NN$. Then
\[
(2n+1)\mid \left(\frac{(2n)!}{2^n}+2^n \right)
\quad\Longleftrightarrow\quad
2n+1 \text{ is a prime number}.
\]
\end{theorem}

\begin{proof}
Since $\gcd(2n+1,2^n)=1$, by Euclid's lemma we have \begin{equation}\label{eq6800}
(2n+1)\mid \left(\frac{(2n)!}{2^n}+2^n\right)
\quad\Longleftrightarrow\quad
(2n+1) \mid ((2n)! + 2^{2n}).
\end{equation} Assume that $2n+1$ is a prime number. By Wilson's theorem, $(2n)! \equiv -1 \pmod{2n+1}$. Since $2n+1\neq 2$, by Fermat's little theorem, $2^{2n} \equiv 1 \pmod{2n+1}$. It follows that 
\[
(2n)! + 2^{2n} \equiv -1 + 1 \equiv 0 \pmod{2n+1}.
\] Thus, $(2n+1) \mid ((2n)! + 2^{2n})$. By \eqref{eq6800}, $(2n+1)\mid \left(\frac{(2n)!}{2^n}+2^n\right)$.

Conversely, assume that
$(2n+1)\mid \left(\frac{(2n)!}{2^n}+2^n\right)$. By \eqref{eq6800}, \begin{equation}\label{eeqazsf}
(2n+1)\mid ((2n)! + 2^{2n}).
\end{equation} Suppose that $2n+1$ is composite. Then $2n+1$ has a prime divisor $p\leq \sqrt{2n+1}\leq 2n$.
In particular, $p \mid (2n)!$ and therefore $(2n)! \equiv 0 \pmod p$. Since $2n+1$ is odd, $p\neq 2$ and therefore $2^{2n}\not \equiv 0 \pmod p$. Thus, $(2n)! + 2^{2n} \not \equiv 0 \pmod p$, meaning that 
\begin{equation}\label{neqa1}
p \nmid ((2n)! + 2^{2n}).
\end{equation} Since $p\mid (2n+1)$, it follows from \eqref{eeqazsf} that  $p \mid ((2n)! + 2^{2n})$, contradicting \eqref{neqa1}.
\end{proof}

\subsection{ \texorpdfstring{\seqnum{A007598}}{}}

The statement of Theorem \ref{t17598} below was conjectured by Bala in \seqnum{A007598}. We shall need several identities concerning the Fibonacci numbers and the Lucas numbers (e.g.,  \cite[Identity 103 on p.\ 112 and Corollary 5.5]{K}).

\begin{theorem}
\begin{enumerate}
\item Let $k\in\ZZ$ and $\ell\in\NN$. Then
\begin{equation}\label{e2467598}
F_{k+\ell}=F_kL_\ell-(-1)^\ell F_{k-\ell}.
\end{equation}   
\item Let $k\in\NN$. Then
\begin{equation}\label{e2467599}
F_{2k} = F_k L_k. 
\end{equation}
\end{enumerate}
\end{theorem}
The following identity is due to Catalan (e.g., \cite[Theorem 5.11]{K}). It is originally stated for $s,t\in\NN$ with $s\geq t$, but with the above extension of the Fibonacci numbers to the negative integers, the identity may be shown to hold true for arbitrary $s,t\in\ZZ$.

\begin{theorem}
Let $s,t\in\ZZ$. Then \begin{equation}\label{catalan}
F_{s+t}F_{s-t}-F_s^{2}=(-1)^{s+t+1}F_t^{2}. 
\end{equation}
\end{theorem}

Finally, we shall need the following result. 

\begin{lemma}\label{l17597}
Let $s\in\NN$ be odd and let $m\in\ZZ$ such that
$m,m+s,m+2s\neq 0$. Then
\[
\frac{1}{F_m F_{m+2s}}
=
\frac{F_s}{F_{2s}}
\left(
\frac{1}{F_m F_{m+s}}
-
\frac{1}{F_{m+s}F_{m+2s}}
\right).
\]
\end{lemma}

\begin{proof}
Taking $k=m+s$ and $\ell=s$ in \eqref{e2467598} yields
\[
F_{m+2s} = F_{m+s}L_s + F_m.
\] Multiplying the equation by $F_s$ and rearranging, we obtain
\[
F_sL_sF_{m+s}=F_s(F_{m+2s}-F_m).
\]
By \eqref{e2467599}, $F_{2s}=F_s L_s$. Thus, 
\[
F_{2s}F_{m+s}=F_s(F_{m+2s}-F_m).
\]
Dividing the equation by $F_{2s}F_mF_{m+s}F_{m+2s}$, the assertion follows.
\end{proof}

\begin{theorem}\label{t17598}
Let $k\in\NN_0$. Then
\begin{equation}\label{e337598}
\sum_{\substack{n\in\NN \\ n\neq 2k+1}}
\frac{1}{F_n^2+(-1)^n F_{2k+1}^2}
=
\frac{1}{F_{4k+2}^2}.
\end{equation}
\end{theorem}

\begin{proof}
Let $S_k$ denote the sum on the left-hand side of \eqref{e337598} and let $n\in\NN$.
Using \eqref{catalan} with $s=n$ and $t=2k+1$, we obtain 
\begin{equation}\label{e17598}
F_{n}^{2}+(-1)^{n}F_{2k+1}^{2}=F_{n+2k+1}F_{n-2k-1}.
\end{equation} Thus,
\[
S_k=\sum_{\substack{n\in\NN \\ n\neq 2k+1}}
\frac{1}{F_{n+2k+1}F_{n-2k-1}}=\sum_{\substack{m\ge -2k\\ m\neq 0}}
\frac{1}{F_m F_{m+4k+2}}.
\] Using Lemma~\ref{l17597} with $s=2k+1$, we obtain
\begin{equation}\label{eee17598}
S_k=\frac{F_{2k+1}}{F_{4k+2}}\sum_{\substack{m\ge -2k\\ m\neq 0}}
\left(
\frac{1}{F_mF_{m+2k+1}}
-
\frac{1}{F_{m+2k+1}F_{m+4k+2}}
\right).
\end{equation} Since $|F_n|\approx \phi^n$ with $\phi=\frac{1+\sqrt{5}}{2}>1$, each of the two sums \[\sum_{\substack{m\ge -2k\\ m\neq 0}}
\frac{1}{F_mF_{m+2k+1}}, \qquad
\sum_{\substack{m\ge -2k\\ m\neq 0}}
\frac{1}{F_{m+2k+1}F_{m+2(2k+1)}}
\] converges absolutely. For $m\in\ZZ\setminus\{0,-2k-1\}$ set
$A_m = \frac{1}{F_mF_{m+2k+1}}$ and let
\[
D = \{m\in\ZZ \;:\; m\ge -2k, m\neq 0\},
\qquad
D' = \{m+2k+1 \;:\; m\in D\}.
\] Notice that $D'=\NN\setminus\{2k+1\}$ and therefore
\[D\setminus D'=\{i\in\ZZ\;:\;1\leq -i\leq 2k\}\cup\{2k+1\}.\]
It follows that
\[
S_k=
\frac{F_{2k+1}}{F_{4k+2}}\left(\sum_{m\in D}A_m - \sum_{m\in D'}A_m\right)
=\frac{F_{2k+1}}{F_{4k+2}}\left(\sum_{i=1}^{2k} A_{-i} + A_{2k+1}\right).
\]
We claim that $\sum_{i=1}^{2k}A_{-i}=0$. Indeed, let $i\in [2k]$. It is straightforward to verify that $A_{-i}=-A_{-(2k+1-i)}$. Thus, the terms in the sum $\sum_{i=1}^{2k}A_{-i}$ cancel each other and the sum is $0$. It follows that
\[S_k = \frac{F_{2k+1}}{F_{4k+2}}A_{2k+1}=\frac{F_{2k+1}}{F_{4k+2}}\frac{1}{F_{2k+1}F_{4k+2}}=\frac{1}{F_{4k+2}^2},\] and the proof is complete.
\end{proof}

\subsection{ \texorpdfstring{\seqnum{A072894}, \seqnum{A073504}, and \seqnum{A098725}}{}}

Let $n\in\NN$. Define an integer sequence $(c^{(n)}_k)_{k\in\NN}$ by
$c^{(n)}_1=1, c^{(n)}_2=n$, and, for $k\in\NN$,
\[
c^{(n)}_{k+2}=
\begin{cases}
\dfrac{c^{(n)}_{k+1}+c^{(n)}_{k}}{2}, & \textnormal{if $c^{(n)}_{k+1}\equiv c^{(n)}_{k}\pmod 2$},\\
\dfrac{c^{(n)}_{k+1}+c^{(n)}_{k}+1}{2}, & \textnormal{if $c^{(n)}_{k+1}\not\equiv c^{(n)}_{k}\pmod 2$}.
\end{cases}
\]
Set $a_n=\lim_{k\to\infty} c^{(n)}_k$. The statements of Theorems \ref{Cloitre2} and \ref{Cloitre3} were conjectured by Cloitre in \seqnum{A072894}.

\begin{lemma}\label{l1894}
\begin{enumerate}
\item Let $x,y\in\ZZ$. Then
\begin{equation}\label{mod2}
\left\lceil \frac{x+y}{2}\right\rceil=
\begin{cases}
\frac{x+y}{2}, & \textnormal{if $x\equiv  y\pmod 2$},\\
\frac{x+y+1}{2}, & \textnormal{if $x\not\equiv  y\pmod 2$}.
\end{cases}
\end{equation}
\item  Let $x\in\ZZ$ and $r\in\{0,1,2\}$. Then
\begin{equation}\label{eq2}
\left\lceil \frac{x+r}{2}\right\rceil-\left\lceil \frac{x}{2}\right\rceil \in \{0,1\}.
\end{equation}
\end{enumerate}
\end{lemma}

\begin{lemma}\label{cc0}
For every $n\in\NN$ the limit $\lim_{k\to\infty} c^{(n)}_k$ exists.
\end{lemma}

\begin{proof}
For $k\in\NN$ set 
$c_k=c^{(n)}_k$ and 
$\Delta_k=|c_k-c_{k+1}|$. First, we observe that, by \eqref{mod2}, 
\begin{equation}\label{c1894}
c_{k+2}=\left\lceil \frac{c_{k+1}+c_{k}}{2}\right\rceil.
\end{equation}
Thus, 
\begin{equation}\label{ee1}
c_{k+2}-c_{k+1}
=\left\lceil \frac{c_{k+1}+c_k}{2}\right\rceil-c_{k+1}
=\left\lceil \frac{c_k-c_{k+1}}{2}\right\rceil.
\end{equation}
We now distinguish between three possibilities for $\Delta_k$.
\begin{enumerate}
\item $\Delta_k=0$. From \eqref{ee1} it follows that $\Delta_{k+1}=0$ as well.
\item $\Delta_k=1$. If $c_k-c_{k+1}=-1$, then by \eqref{ee1}, $\Delta_{k+1}=0$.
If $c_k-c_{k+1}=1$, then by \eqref{ee1},  $c_{k+2}-c_{k+1}=1$. Applying this case with $k+1$ yields $\Delta_{k+2}=0$.
\item $\Delta_k\ge 2$. If $c_k-c_{k+1}\leq -2$, then $\frac{c_k-c_{k+1}}{2}\leq -1$. By \eqref{ee1},
\[
\Delta_{k+1}=|c_{k+1}-c_{k+2}|=\left|\left\lceil \frac{c_k-c_{k+1}}{2}\right\rceil\right|
=\left\lfloor \frac{c_{k+1}-c_k}{2}\right\rfloor
\leq c_{k+1}-c_k-1
=\Delta_k-1.
\]
If $c_k-c_{k+1}\geq 2$, then $\frac{c_{k}-c_{k+1}}{2}\ge 1$. By \eqref{ee1},
\[
\Delta_{k+1}=|c_{k+1}-c_{k+2}|=\left|\left\lceil \frac{c_{k}-c_{k+1}}{2}\right\rceil\right|
\leq c_{k}-c_{k+1}-1
=\Delta_k-1.
\]
\end{enumerate}
By definition, $\Delta_1=|c_2-c_1|=|n-1|$. By the third case, if $\Delta_k\ge 2$ then $\Delta_{k+1}<\Delta_k$. Thus, for some $k\in\NN$, we have $\Delta_k\in\{0,1\}$. From the first case it follows that if $\Delta_k=0$, then $\Delta_\ell=0$ for every $\ell>k$. If $\Delta_k=1$, by the second case, $\Delta_{k+2}=0$. Thus, $\Delta_\ell=0$ for every $\ell>k+1$. We conclude that the sequence $(\Delta_k)_{k\in\NN}$ is eventually zero and therefore $(c_k)_{k\in\NN}$ is eventually constant. In particular, the limit $a_n=\lim_{k\to\infty} c_k$ exists.
\end{proof}

\begin{theorem}\label{Cloitre2}
Let $n\in\NN$. Then $a_{n+1}-a_n\in\{0,1\}$.
\end{theorem}

\begin{proof}
For $k\in\NN$ set $\delta_k=c^{(n+1)}_k-c^{(n)}_k$. We prove by induction that
$\delta_k,\delta_{k+1}\in\{0,1\}$ for every $k\in\NN$. We have \[\delta_1=1-1=0,\qquad  \delta_2=n+1-n=1.\] 
Now assume that $\delta_k,\delta_{k+1}\in\{0,1\}$ for some $k\in\NN$. Thus, $\delta_k+\delta_{k+1}\in\{0,1,2\}$. By \eqref{c1894} and \eqref{eq2},
\begin{align*}
\delta_{k+2}&=c^{(n+1)}_{k+2}-c^{(n)}_{k+2}\\
&=\left\lceil \frac{c_{k+1}^{(n+1)}+c_{k}^{(n+1)}}{2}\right\rceil-\left\lceil \frac{c_{k+1}^{(n)}+c_{k}^{(n)}}{2}\right\rceil\\
&=\left\lceil \frac{c_{k+1}^{(n)}+c_{k}^{(n)}+\delta_{k}+\delta_{k+1}}{2}\right\rceil-\left\lceil \frac{c_{k+1}^{(n)}+c_{k}^{(n)}}{2}\right\rceil\in\{0,1\}.
\end{align*}
This concludes the proof of the induction. It follows that
\[a_{n+1}-a_n = \lim_{k\to\infty} c_k^{(n+1)}-\lim_{k\to\infty} c_k^{(n)}=\lim_{k\to\infty} (c_k^{(n+1)}-c_k^{(n)})=\lim_{k\to\infty} \delta_k.\] Since $\delta_k\in\{0,1\}$ for every $k\in\NN$, necessarily $\lim_{k\to\infty} \delta_k\in\{0,1\}$, and the proof is complete.
\end{proof}

\begin{theorem}\label{Cloitre3}
We have
\[
\lim_{n\to\infty}\frac{a_n}{n}=\frac{2}{3}.
\]
\end{theorem}

\begin{proof}
For $n\in\NN_0$ let $s_n$ denote the number of $1$'s in the base $-2$ expansion of $n$ (negabinary base). By a formula in \seqnum{A027615},
\begin{equation}\label{eq27651}
\frac{a_n}{n}=\frac{s_{n-1}+1}{3n}+\frac{2}{3}.
\end{equation}
By \cite[p.~209]{GT}, the number of digits in the base $-2$ representation of $n$ differs by at most three from the number of digits in its ordinary binary representation. Hence,
\[
s_n\leq \lfloor\log_2 n\rfloor+4=O(\log n).\qedhere
\]
\end{proof}

In \seqnum{A072894} Cloitre suggests that for $n>1000$,
\[\frac{1}{2} < \frac{3a_n-2n}{\log n} <\frac{3}{2}.\]
Furthermore, Cloitre asks whether 
\[\lim_{n\to\infty} \frac{3a_n-2n}{\log n}=1.\]
The statement of the following lemma refutes the conjecture and settles the question in the negative.

\begin{lemma}
For $t\in\NN$ let $n_t=4^t+1$. Then
\[
\lim_{t\to\infty}\frac{3a_{n_t}-2n_t}{\log n_t}=0.
\]
\end{lemma}

\begin{proof}
Let $t\in\NN$. Then $s_{n_t-1}=s((-2)^{2t})=1$. By \eqref{eq27651},
\[
3a_{n_t}-2n_{t}=s_{n_t-1}+1=2.\]
Thus, 
\[
\lim_{t\to\infty}\frac{3a_{n_t}-2n_t}{\log n_t}=\lim_{t\to\infty}\frac{2}{\log(4^t+1)}= 0.\qedhere
\]
\end{proof}

Let $n\in\NN$. Define an integer sequence $(u^{(n)}_k)_{k\in\NN}$ by
$u^{(n)}_1=1, u^{(n)}_2=n$, and, for $k\in\NN$,
\begin{equation}\label{largeenought}
u^{(n)}_{k+2}=\left\lfloor\frac{u^{(n)}_{k+1}}{2}\right\rfloor
+\left\lfloor\frac{u^{(n)}_{k}}{2}\right\rfloor.
\end{equation}
Set $U_n=\lim_{k\to\infty} u^{(n)}_k$. The $n$th element of the sequence \seqnum{A073504} is defined to be $U_n$. 

\begin{lemma}\label{lemUexists}
For every $n\in\NN$ the limit $\lim_{k\to\infty} u^{(n)}_k$ exists.
Thus, $U_n$ is well-defined. Furthermore, $U_n$ is even.
\end{lemma}

\begin{proof}
It follows immediately from the definition of the sequence $(u^{(n)}_k)_{k\in\NN}$ that $u_k^{(n)}\in\NN_0$, for every $k,n\in\NN$. Now let $n\in\NN$. For $k\in\NN$ set $u_k=u_k^{(n)}$ and define $M_k=\max\{u_k,u_{k+1}\}$.
We have
\[u_{k+2}
=
\left\lfloor\frac{u_{k+1}}{2}\right\rfloor
+
\left\lfloor\frac{u_k}{2}\right\rfloor
\leq
\frac{u_{k+1}+u_k}{2}
\leq M_k.
\]
Since also $u_{k+1}\leq M_k$, it follows that
\[
M_{k+1}
=
\max\{u_{k+1},u_{k+2}\}
\leq M_k.
\]
Thus, $(M_k)_{k\in\NN}$ is a nonincreasing sequence of nonnegative
integers and is therefore eventually constant. Let $M\in\NN_0$ and
$k\in\NN$ be such that $M_\ell=M$, for every $\ell\in\NN$ with $\ell\geq k$. We claim that $u_{\ell+1}=M$ for every $\ell\geq k$. Indeed, assuming otherwise, there is $m\in\NN$ with $m\geq k$ such that $u_{m+1}<M$.
Since $M_m=\max\{u_m,u_{m+1}\}=M$, necessarily $u_m=M$. Similarly, since
$M_{m+1}=\max\{u_{m+1},u_{m+2}\}=M$,
necessarily $u_{m+2}=M$. On the other hand, since $u_{m+1}\leq M-1$, we have
\[
u_{m+2}
=
\left\lfloor\frac{u_{m+1}}2\right\rfloor
+
\left\lfloor\frac{u_m}2\right\rfloor
\leq
\left\lfloor\frac{M-1}{2}\right\rfloor
+
\left\lfloor\frac{M}{2}\right\rfloor
=M-1,
\]
contradicting $u_{m+2}=M$. Thus,
$u_{\ell+1}=M$ for every $\ell\geq k$ and therefore $U_n=\lim_{k\to\infty}u_k=M$. Taking $k$ large enough in \eqref{largeenought}, we obtain
\[
U_n=2\left\lfloor\frac{U_n}2\right\rfloor.
\]
Thus, $U_n$ is even.
\end{proof}

\begin{lemma}\label{lem9894}
Let $k,n\in\NN$ with $n\geq 2$. Then
\begin{equation}\label{kbhdsf1}
c^{(n)}_{k+1}+\left\lfloor \frac{u_k^{(n-1)}}{2}\right\rfloor=n.
\end{equation}
\end{lemma}

\begin{proof}
We proceed by induction on $k$. Using \eqref{c1894}, we have 
\begin{align*}
c^{(n)}_{2}+\left\lfloor\frac{u^{(n-1)}_1}{2}\right\rfloor
&=n+\left\lfloor\frac{1}{2}\right\rfloor
=n,\\
c^{(n)}_3+\left\lfloor\frac{u^{(n-1)}_2}{2}\right\rfloor
&=\left\lceil\frac{n+1}{2}\right\rceil+\left\lfloor\frac{n-1}{2}\right\rfloor=n.
\end{align*} Thus, the assertion holds for $k=1,2$. Assume now that the assertion holds for $k-1$ and for $k$, for some $k\in\NN$ with $k\geq 2$. Using \eqref{c1894} and the induction hypothesis, we have
\begin{align*}
c^{(n)}_{k+2}
&=\left\lceil \frac{c^{(n)}_{k+1}+c^{(n)}_k}{2}\right\rceil \\
&=\left\lceil \frac{n-\left\lfloor \frac{u_{k}^{(n-1)}}{2}\right\rfloor +n-\left\lfloor \frac{u_{k-1}^{(n-1)}}{2}\right\rfloor }{2}\right\rceil \\
&=\left\lceil n-\frac{u^{(n-1)}_{k+1}}{2}\right\rceil \\
&=n-\left\lfloor\frac{u^{(n-1)}_{k+1}}{2}\right\rfloor,
\end{align*}
completing the proof.
\end{proof}

\begin{theorem}\label{tm1894}
Let $n\in\NN$ with $n\geq 2$. Then
$a_n=n-\frac{U_{n-1}}{2}$.
\end{theorem}

\begin{proof}
Since the sequences $(c_k^{(n)})_{k\in\NN}$ and $(u_k^{(n)})_{k\in\NN}$ are eventually constant, taking $k$ large enough in \eqref{kbhdsf1} yields $a_n=n-\left\lfloor\frac{U_{n-1}}{2}\right\rfloor$. Since $U_{n-1}$ is even, the assertion follows.
\end{proof}

The statement of the following corollary was conjectured by Calder\'{o}n in \seqnum{A073504}.

\begin{corollary}
Let $n\in\NN$. Then $U_n=\frac{2}{3}(n-s_n)$.
\end{corollary}

\begin{proof}
By Theorem \ref{tm1894}, we have $a_{n+1}=(n+1)-\frac{U_n}{2}$. By \eqref{eq27651}, $a_{n+1}=\frac{s_n+2n}{3}+1$. Equating both expressions for $a_{n+1}$ and solving for $U_n$, the assertion follows.
\end{proof}

The statement of the following corollary was conjectured in \seqnum{A098725}.

\begin{corollary} 
Let $(b_n)_{n\in\NN_0}$ be defined by \[ b_{4n}=0,\qquad b_{2n+1}=1,\qquad b_{4n+2}=b_{n+1}, \qquad n\in\NN_0. \] 
Then, for every $n\in\NN$, 
$b_n=a_{n+1}-a_n$. 
\end{corollary} 

\begin{proof} 
Let $d_0=0$ and, for $n\in\NN$ set $d_n=a_{n+1}-a_n$. By \eqref{eq27651}, \begin{equation}\label{eqqdiffer} 
d_n=\frac{s_n-s_{n-1}+2}{3}. 
\end{equation} 
We claim that, for every $m\in\NN_0$, 
\begin{align*}
s_{4m}&=s_{m},\\ s_{4m+1}&=s_{m}+1,\\ s_{4m+2}&=s_{m+1}+1,\\ s_{4m+3}&=s_{m+1}+2. \end{align*} Indeed, multiplication by $4=(-2)^2$ shifts a negabinary expansion by two positions. Furthermore, 
\begin{align*}
4m+1&=4m+1\cdot(-2)^0,\\
4m+2&=4(m+1)+1\cdot(-2)^1,\\
4m+3&=4(m+1)+1\cdot (-2)^1+1\cdot (-2)^0.
\end{align*} It follows that for $m\in\NN$, equation \eqref{eqqdiffer} yields 
\[
d_{4m} = \frac{s_{m}-(s_{m}+2)+2}{3} = 0. 
\] Similarly, for $m\in\NN_0$, 
\begin{align*}
d_{4m+1}&=\frac{s_{m}+1-s_{m}+2}{3}=1,\\
d_{4m+2}&=\frac{s_{m+1}+1-(s_{m}+1)+2}{3}=d_{m+1},\\
d_{4m+3}&=\frac{s_{m+1}+2-(s_{m+1}+1)+2}{3}=1.
\end{align*} From the identities for $d_{4m+1}$ and $d_{4m+3}$ it follows that $d_{2m+1}=1$. Thus, the sequence $(d_n)_{n\in\NN_0}$ satisfies the same recurrence as the sequence $(b_n)_{n\in\NN_0}$. Since 
$d_0=0=b_0$ and $d_1=a_2-a_1=1=b_1$, the two sequences coincide.
\end{proof}

\subsection{ \texorpdfstring{\seqnum{A078111}}{}}

For $n\in\NN$ set $x_n=\frac{(n+2)^{n+2}}{n^n}$.
Greathouse IV conjectured in \seqnum{A078111} that \[x_n = e^2\left(n^2 + 2n + \frac{2}{3} + o(1)\right).\]
It follows from the following theorem that the conjecture is true.

\begin{theorem}\label{tt12}
Let $n\in\NN$. Then \begin{equation}\label{e1111} e^2\left((n+1)^2-\frac13-\frac{2}{45n^2}\right) < x_n < e^2\left((n+1)^2-\frac13\right). \end{equation} 
\end{theorem} 

\begin{proof}
For $t>0$, define
\[
F(t)=
\ln\left(1+t+\frac{t^2}{6}\right)
-\left(\frac{2}{t}+2\right)\ln(1+t)+2.
\]
We claim that $F(t)>0$ for $t>0$. Indeed,
\[
F'(t)
=
\frac{2H(t)}
{t^2(t^2+6t+6)},
\]
where
\[
H(t)
=
(t^2+6t+6)\ln(1+t)-3t^2-6t.
\] We have 
\begin{align*} 
H'(t) &=(2t+6)\ln(1+t)-5t-\frac{t}{1+t},\\ 
H''(t) &=2\ln(1+t)-3+\frac{4}{1+t}-\frac{1}{(1+t)^2},\\
H'''(t) &=\frac{2t^2}{(1+t)^3}. 
\end{align*}
Thus,
$H(0)=H'(0)=H''(0)=0$ and 
$H'''(t)>0$ for $t>0$. It follows successively that $H(t)>0$ for every $t>0$ and therefore $F'(t)>0$. Since $F$ extends continuously to $t=0$ with $F(0)=0$, we conclude that $F(t)>0$ for every $t>0$. Therefore,
\[
\ln\left(1+t+\frac{t^2}{6}\right)
>
\left(\frac{2}{t}+2\right)\ln(1+t)-2.
\]
Exponentiating gives
\[
(1+t)^{2/t+2}<e^2\left(1+t+\frac{t^2}{6}\right).
\]
Substituting $t=\frac{2}{n}$ and multiplying by $n^2$, we obtain
\[
x_n<e^2\left(n^2+2n+\frac23\right)=e^2\left((n+1)^2-\frac13\right).
\]

We now prove the lower bound. For $0<t\le2$, define
\[
J(t)=
\left(\frac{2}{t}+2\right)\ln(1+t)-2
-\ln\left(1+t+\frac{t^2}{6}-\frac{t^4}{360}\right).
\]
Notice that 
\[1+t+\frac{t^{2}}{6}-\frac{t^{4}}{360}\geq1+t+\left(\frac{1}{6}-\frac{4}{360}\right)t^{2}=1+t+\frac{7t^{2}}{45}>0.\] Thus, the logarithm in $J(t)$ is well defined. A direct
calculation gives
\[
J'(t)
=
\frac{2K(t)}
{t^2(360+360t+60t^2-t^4)},
\]
where
\[ K(t) = t^5+180t^2+360t + (t^4-60t^2-360t-360)\ln(1+t). \]
We have
\begin{align*}
K'(t)
&=
\frac{5t^5+6t^4+300t^2+360t+4(t^4+t^3-30t^2-120t-90)\ln(1+t)
}{(1+t)},\\
K''(t)
&=
\frac{20t^5+47t^4+28t^3+180t^2+120t+12(t^4+2t^3-9t^2-20t-10)\ln(1+t)}{(1+t)^2},\\
K'''(t)
&=
\frac{2t(30t^4+103t^3+120t^2-12t+12(1+t)^3\ln(1+t))
}{(1+t)^3},\\
K^{(4)}(t)
&=
\frac{
2(60t^5+265t^4+448t^3+408t^2-12t+12(1+t)^4\ln(1+t))
}{(1+t)^4},\\
K^{(5)}(t)
&=
\frac{24t(5t^4+26t^3+55t^2+50t+75)}{(1+t)^5}.
\end{align*} Thus,
$K^{(j)}(0)=0$ for $j=0,1,2,3,4$, and $K^{(5)}(t)>0$ for every $t>0$. It follows successively that 
$K(t)>0$, for every $t>0$. Furthermore, for $0<t\leq 2$ we have
\[
360+360t+60t^{2}-t^{4}\geq360+360t+56t^{2}>0.
\]
It follows that $J'(t)>0$ for $0<t\leq 2$. Since $J$ extends continuously
to $t=0$ with $J(0)=0$, it follows that $J(t)>0$
for every $0<t\le2$. Consequently,
\[
\left(\frac{2}{t}+2\right)\ln(1+t)-2>\ln\left(1+t+\frac{t^2}{6}-\frac{t^4}{360}\right).
\]
Exponentiating, substituting $t=\frac{2}{n}$, and multiplying by $e^2n^2$, we obtain
\[
x_n>e^2\left(n^2+2n+\frac23-\frac{2}{45n^2}\right)
=e^2\left((n+1)^2-\frac13-\frac{2}{45n^2}\right).\qedhere
\]
\end{proof}

\begin{remark}
A stronger conjecture recorded in \seqnum{A078111} asserts that for $n$ large enough,
\begin{equation}\label{eqqqn2}
\left\lfloor \frac{(n+2)^{n+2}}{n^n}\right\rfloor
=
\left\lfloor
e^2\left((n+1)^2-\frac13\right)
\right\rfloor.
\end{equation}
Theorem \ref{tt12} does not establish this stronger statement. Indeed, if
\[
y_n=e^2\left((n+1)^2-\frac13\right),
\]
then \eqref{e1111} gives
\[
0<y_n-x_n<\frac{2e^2}{45n^2}.
\]
Thus, the equality in \eqref{eqqqn2} would follow from the additional estimate
\begin{equation}\label{yn2}
\{y_n\}>\frac{2e^2}{45n^2},
\end{equation}
where $\{y\}$ is the fractional part of $y$. We have verified that \eqref{yn2} holds for every $3\leq n\leq 10^8$ and that \eqref{eqqqn2} holds for every $1\leq n\leq 10^8$.
\end{remark}

\subsection{\texorpdfstring{\seqnum{A102309}}{}}

The statement of Theorem \ref{Steph01} below was conjectured by Stephan in \seqnum{A102309}. We shall use the following well-known result (e.g., \cite[Theorem 6.6]{B}). 
\begin{lemma}\label{muu1}
Let $n\in\NN$. Then 
\[\sum_{d\mid n} \mu(d) = 
\mathbf{1}_{n=1}.\]
\end{lemma} 

\begin{theorem}\label{Steph01}
For $n\in\NN$ let
\begin{align*}
a_n& = \sum_{d\mid n} \mu(d)\binom{\frac{n}{d}}{2},\\ 
b_n& = \#\{(i,j,k) \in [n]^3\;:\; i\leq j < k \textnormal{ and } \gcd(i,j,k) = 1\}.
\end{align*}
Then, for $n\geq 2$ we have $a_{n}=b_n-b_{n-1}$.
\end{theorem}

\begin{proof}
Using Lemma \ref{muu1}, 
\begin{align*}
b_{n}-b_{n-1}
&=
\#\{(i,j,k)\in[n]^3\;:k=n,\ i\leq j<n, \textnormal{ and } \gcd(i,j,n)=1\}\\
&=\sum_{1\leq i\leq j<n}\mathbf{1}_{\gcd(i,j,n)=1}\\
&=\sum_{1\leq i\leq j<n}\sum_{d\mid \gcd(i,j,n)}\mu(d)\\
&=\sum_{d\mid n}\mu(d)
\#\{(i,j)\in[n]^2\;:\; i\leq j<n\textnormal{ and } d\mid \gcd(i,j)\}.
\end{align*} 
Let $d\mid n$ and write
$i=dr,j=ds$, for some $r,s\in\NN$. Then
\[
i\leq j<n \textnormal{ and } d\mid \gcd(i,j) \iff 
r\leq s<\frac{n}{d}\iff 
r\leq s\leq \frac{n}{d}-1.\] Thus,
\begin{align*}
b_{n}-b_{n-1}&=\sum_{d\mid n}\mu(d)
\#\left\{(r,s)\in\left[\frac{n}{d}-1\right]^2\;:\; r\leq s\right\}\\
&=\sum_{d\mid n}\mu(d)
\binom{\left(\frac{n}{d}-1\right)+1}{2}\\
&=a_{n}.\qedhere 
\end{align*} 
\end{proof}

\subsection{\texorpdfstring{\seqnum{A111386}}{}}

Let $x,y \in\NN$. We let $[x\circ y]$ denote the base-$10$ concatenation of $x$ and $y$, i.e., 
\[
[x\circ y]=x\cdot 10^{\digitlen{y}}+y.
\] Let $(a_n)_{n\in\NN}$ be the sequence defined as follows: Set $a_1=1, a_2=3$, and, for $n\in\NN$ with $n \ge 3$, let $a_n$ be the smallest odd positive integer not among
$\{a_1,\dots,a_{n-1}\}$ such that $a_{n-1} \mid [a_{n-2}\circ a_n]$.  The statement of the following theorem was conjectured by Brockhaus in \seqnum{A111386}.

\begin{theorem}
Let $k\in\NN$. Then
$a_{2k-1}=5^{k-1}$ and $a_{2k}=3\cdot 5^{k-1}$.
\end{theorem}

\begin{proof}
We prove the assertion by strong induction on $k\in\NN$. We have $a_1=1=5^{0}$ and $a_2=3=3\cdot 5^{0}$. Thus, the assertion holds for $k=1$. Assume that the assertion holds for every $j\in[k]$ for some $k\in\NN$. By the induction hypothesis, $a_{2k-1}=5^{k-1}$ and $a_{2k}=3\cdot 5^{k-1}$. Thus, $a_{2k+1}$ is the smallest odd positive integer not among $\{a_1,\dots,a_{2k}\}$ such that
\[3\cdot 5^{k-1} \mid (5^{k-1}\cdot 10^{\digitlen{a_{2k+1}}}+a_{2k+1}).\] In particular, $5^{k-1}\mid a_{2k+1}$. Thus, $a_{2k+1}=m\cdot 5^{k-1}$ for some odd $m\in\NN$. Since $a_{2k-1}=1\cdot 5^{k-1}$ and $a_{2k}=3\cdot 5^{k-1}$, necessarily $a_{2k+1}\ge 5\cdot 5^{k-1}=5^{k}$. We show that $a_{2k+1}=5^k$. First, $5^k>3\cdot 5^{k-1}= \max\{a_1,\ldots,a_{2k}\}$.
Second, \[
[5^{k-1}\circ 5^{k}]
=5^{k-1}\cdot 10^{\digitlen{5^{k}}}+5^{k}
=5^{k-1}(10^{\digitlen{5^{k}}}+5).
\] Thus, $5^{k-1}\mid [5^{k-1}\circ 5^{k}]$. Moreover, $ 10^{\digitlen{5^k}}+5\equiv 1+2\equiv 0\pmod 3$. Consequently, 
\[3\cdot5^{k-1}\mid 5^{k-1}\left(10^{\digitlen{5^k}}+5\right) =[5^{k-1}\circ5^k].
\] Hence, all the defining conditions are fulfilled, and therefore $a_{2k+1}=5^k$.

We now consider $a_{2k+2}$. By the induction hypothesis and the previous part, $a_{2k}=3\cdot5^{k-1}$ and $a_{2k+1}=5^k$. Thus, $a_{2k+2}$ is the smallest odd positive integer not among $\{a_1,\dots,a_{2k+1}\}$ such that 
\[
5^{k}\mid (3\cdot 5^{k-1}\cdot 10^{\digitlen{a_{2k+2}}}+a_{2k+2}).
\]
Since $\digitlen{a_{2k+2}}\ge 1$, we have $5^k \mid 3\cdot 5^{k-1}\cdot 10^{\digitlen{a_{2k+2}}}$. Thus,
$5^k\mid a_{2k+2}$ and therefore $a_{2k+2}=m\cdot 5^{k}$ for some odd $m\in\NN$.
Since $a_{2k+1}=1\cdot 5^{k}$, necessarily $a_{2k+2}\ge 3\cdot 5^{k}$. We show that $a_{2k+2}=3\cdot5^k$. First, $3\cdot 5^k> 5^{k}= \max\{a_1,\ldots,a_{2k+1}\}$. Second
\[
[3\cdot5^{k-1}\circ3\cdot5^{k}]=3\cdot5^{k-1}\cdot10^{\digitlen{3\cdot5^{k}}}+3\cdot5^{k}=3\cdot5^{k-1}(10^{\digitlen{3\cdot5^{k}}}+5).\]
Since $\digitlen{3\cdot 5^k}\geq 1$, we have $5\mid (10^{\digitlen{3\cdot 5^{k}}}+5)$. Thus, $5^k\mid [3\cdot 5^{k-1}\circ 3\cdot 5^{k}]$. Hence, all conditions are fulfilled and therefore $a_{2k+2}=3\cdot 5^{k}$. This proves the assertion for $k+1$ and the induction is complete.
\end{proof}

\subsection{\texorpdfstring{\seqnum{A116098} and \seqnum{A116129}}{}}

Sequence \seqnum{A116098} consists of the integers $k\geq 10$ for
which $[k\circ(k-9)]$ can be expressed as the product of two positive
integers differing by $6$. As observed by Israel, this is equivalent
to requiring that $k(10^{\digitlen{k-9}}+1)$ be a square. Indeed, if the two factors are $m$ and $m+6$, then
\[
[k\circ(k-9)]
=k\cdot 10^{\digitlen{k-9}}+k-9
=m(m+6).
\]
Thus,
\[
k(10^{\digitlen{k-9}}+1)=m(m+6)+9=(m+3)^2.
\]
The converse follows by reversing this argument.

Similarly, sequence \seqnum{A116129} consists of the integers $k\geq 5$
for which $[k\circ(k-4)]$ can be expressed as the product of two
positive integers differing by $4$. Israel observed that this is equivalent to requiring that
$k(10^{\digitlen{k-4}}+1)$ be a square. The proof is similar to the one given above. The following theorem was conjectured by Israel in \seqnum{A116098} and \seqnum{A116129}.

\begin{theorem}
The sequences \seqnum{A116098} and \seqnum{A116129} coincide.
\end{theorem}

\begin{proof}
``$\seqnum{A116098}\subseteq \seqnum{A116129}$". Let $k\in \seqnum{A116098}$ and set $d=\digitlen{k-9}$. Then $k(10^d+1)$ is a square. To see that $k\in \seqnum{A116129}$, it suffices to show that $\digitlen{k-4}=d$. Clearly, $\digitlen{k-4}\in\{d,d+1\}$. Assume that $\digitlen{k-4}=d+1$. Then $k=10^d+t$ with $t\in\{4,5,6,7,8\}$. Thus,
\begin{equation}\label{el98}
k(10^d+1)=10^{2d}+(t+1)10^d+t.
\end{equation} But for each value of $t$, the number in \eqref{el98} lies strictly between consecutive squares. Indeed,
\begin{itemize}
\item if $t=4$, then $(10^d+2)^2 < 10^{2d}+5\cdot 10^d+4 < (10^d+3)^2$;
\item if $t=5$, then $(10^d+2)^2 < 10^{2d}+6\cdot 10^d+5 < (10^d+3)^2$;
\item if $t=6$, then $(10^d+3)^2 < 10^{2d}+7\cdot 10^d+6 < (10^d+4)^2$;
\item if $t=7$, then $(10^d+3)^2 < 10^{2d}+8\cdot 10^d+7 < (10^d+4)^2$;
\item if $t=8$, then $(10^d+4)^2 < 10^{2d}+9\cdot 10^d+8 < (10^d+5)^2$.
\end{itemize} Thus, $k(10^d+1)$ cannot be a square, a contradiction. Hence $\digitlen{k-4}=d$ and therefore
$k\in\seqnum{A116129}$.

``$\seqnum{A116129}\subseteq \seqnum{A116098}$". Assume that $k\in \seqnum{A116129}$ and set $d=\digitlen{k-4}$. Then $k(10^d+1)=x^2$ is a square. To see that $k\in\seqnum{A116098}$, it suffices to show that $\digitlen{k-9}=d$. We distinguish between three cases.
\begin{enumerate}
\item $d=1$. Then $5\leq k\leq 13$ and 
$k(10^1+1)=11k$ is a square. We have $55\leq 11k\leq 143$ and the only square in this interval, which is divisible by $11$, is $121=11^2$, forcing $k=11$. Then $\digitlen{k-9}=\digitlen{2}=1=d$. 
\item $d=2$. Then $14\leq k\leq 103$ and
$k(10^2+1)=101k$ is a square. We have $1414\leq 101k\leq 10403$ and the only square in this interval, which is divisible by $101$, is $10201=101^2$, forcing $k=101$. Then $\digitlen{k-9}=\digitlen{92}=2=d$.
\item $d\ge 3$. Clearly, $\digitlen{k-9}\in\{d,d-1\}$. Assume that $\digitlen{k-9}=d-1$. Then $k=10^{d-1}+t$ with $t\in\{4,5,6,7,8\}$. Let $x\in\NN$ such that $k(10^d+1)=x^2$ and set $A=10^{d-1}$. Notice that $100\mid A$. We have
\[
x^2 \equiv k(10^d+1) \equiv t \pmod{A}.
\] In particular, \[x^2 \equiv t\pmod{10}, \qquad x^2 \equiv t\pmod{4}, \qquad x^2 \equiv t\pmod{100}.\] Since the set of quadratic residues mod $10$ is $\{0,1,4,5,6,9\}$, necessarily $t\neq 7,8$. Since the set of quadratic residues mod $4$ is $\{0,1\}$, necessarily
$t\neq 6$. Since $5$ is not a quadratic residue mod $100$, necessarily $t\neq 5$. It follows that $t=4$ and therefore $k=A+4$. Thus, 
\begin{equation}\label{ka98}
x^2=k(10^d+1)=(A+4)(10A+1).
\end{equation}
Set $g=\gcd(A+4,10A+1)$. Then \[
g=\gcd(A+4,10(A+4)-39)=\gcd(A+4,39).
\]
Since $A\equiv 1\pmod{3}$, we have  $3\nmid (A+4)$ and therefore
$g\in\{1,13\}$. If $g=1$, then \eqref{ka98} implies that $A+4$ is a square. But $A+4\equiv 2\pmod{3}$ and $2$ is not a quadratic residue mod $3$. On the other hand, if $g=13$, write $A+4=13u_1$ and $10A+1=13v_1$ for some $u_1,v_1\in\NN$ with $\gcd(u_1,v_1)=1$. By \eqref{ka98}, $u_1v_1=(\frac{x}{13})^2$ is a square and therefore $u_1$ is a square as well. Now, $13u_1\equiv A+4\equiv  4\pmod{5}$ and therefore $u_1\equiv  3\pmod{5}$. But $3$ is not a quadratic residue mod $5$. This contradiction shows that $\digitlen{k-9}\neq d-1$. Thus, $\digitlen{k-9}=d$ and we are done. \qedhere
\end{enumerate}
\end{proof}

\subsection{\texorpdfstring{\seqnum{A135418}}{}}

For $n\in\NN$ let $a_{n}$ denote the number of ways to partition the set $[16]$ into two subsets such that sums of the $n$-th powers in the two subsets are equal. The statement of the following theorem was conjectured by Seidov in \seqnum{A135418}.

\begin{theorem}
Let $n\in\NN$ with $n\geq 5$. Then $a_{n}=0$.
\end{theorem}

\begin{proof}
Set $Z_n=\frac12\sum_{k=1}^{16} k^n$. A solution counted by $a_{n}$ corresponds to a subset $A\subseteq [16]$ such that 
$\sum_{a\in A} a^n=Z_n$. Replacing $A$ by its complement if necessary, we may assume that $16\in A$. Set \[R_n=\sum_{a\in A\setminus\{16\}}a^n=Z_n-16^n=\frac{1}{2}\left(\sum_{k=1}^{15}k^n-16^n\right).\] Now,
\[
\sum_{k=1}^{15}k^{11}
=15662165784000
<17592186044416
=16^{11}.
\]
For every $k\in[15]$, the sequence
$n\mapsto \left(\frac{k}{16}\right)^n$ is decreasing. Hence, for every $n\geq11$,
\[
\frac{\sum_{k=1}^{15}k^n}{16^n}
=\sum_{k=1}^{15}\left(\frac{k}{16}\right)^n
\leq
\sum_{k=1}^{15}\left(\frac{k}{16}\right)^{11}
<1.
\]
Thus,
\[
R_n=\frac12\left(\sum_{k=1}^{15}k^n-16^n\right)<0,
\]
contradicting the fact that $R_n$ is a sum of nonnegative integers. It remains to consider the cases $n=5,6,7,8,9,10$.

\begin{enumerate}

\item $n=10$. We have
$R_{10}=3510520212<10^{10}$. Hence $10,11,12,13,14,15\notin A$. Since
$\sum_{k=1}^{8}k^{10}=1427557524<R_{10}$,
necessarily $9\in A$. Now, $R_{10}-9^{10}=23735811$. Since $6^{10}>23735811$, necessarily $6,7,8\notin A$. However,
$\sum_{k=1}^{5}k^{10}=10874275<23735811$.
Thus, $n=10$ is impossible.

\item $n=9$. We have $R_9=5040730912<12^9$. Thus, $12,13,14,15\notin A$. On the other hand,
$\sum_{k=1}^{11}k^9=3932252676<R_9$. Thus, $n=9$ is impossible.

\item $n=8$. We have $R_8=685757508<13^8$. Thus, $13,14,15\notin A$. Since $\sum_{k=1}^{11}k^8=382090214<R_8$, necessarily $12\in A$. Now, $R_8-12^8=255775812$. Since $\sum_{k=1}^{10}k^8=167731333<R_8-12^8$, necessarily $11\in A$. We have $R_8-12^8-11^8=41416931$. Since $9^8>41416931$, necessarily $9,10 \notin A$. On the other hand, $\sum_{k=1}^{8}k^8=24684612<41416931$. Thus, $n=8$ is impossible. 

\item $n=7$. We have $R_7=71992672<14^7$. Thus, $14,15\notin A$. Assume that $13\in A$. We have $R_7-13^7=9244155$. Since
$10^7>9244155$, necessarily $10,11,12\notin A$. On the other hand, $\sum_{k=1}^{9}k^7=8080425<9244155$, showing that $13\in A$ is impossible. Assume now that $13\notin A$. Since $\sum_{k=1}^{11}k^7=37567596<R_7$, necessarily $12\in A$. We have $R_7 - 12^7=36160864$. Since $\sum_{k=1}^{10}k^7=18080425<36160864$, necessarily $11\in A$. Then $R_7 - 12^7-11^7=16673693$. Since $\sum_{k=1}^{9}k^7=8080425<16673693$, necessarily $10\in A$. Then $R_7 - 12^7-11^7-10^7=6673693$. Since $\sum_{k=1}^{8}k^7=3297456<6673693$, necessarily $9\in A$. Then $R_7 - 12^7-11^7-10^7-9^7=1890724$. But $8^7>1890724$ and therefore $8\notin A$. On the other hand, $\sum_{k=1}^{7}k^7=1200304<1890724$, showing that $13\notin A$ is impossible, which, in turn, concludes the proof that $n=7$ is impossible.

\item $n=6$. We have $R_6=6852852<14^6$. Thus, $14,15\notin A$. Since $\sum_{k=1}^{12}k^6=6735950<R_6$, necessarily $13\in A$. We have $R_6-13^6=2026043$. But
$12^6>2026043$ and therefore $12\notin A$. Since $\sum_{k=1}^{10}k^6=1978405<2026043$, necessarily $11\in A$. We have $R_6-13^6-11^6=254482$. But $8^6>254482$ and therefore $8,9,10\notin A$. On the other hand
$\sum_{k=1}^{7}k^6=184820<254482$, concluding the proof that $n=6$ is impossible. 

\item $n=5$. We have $R_5=625312<15^5$. Thus, $15\notin A$. We now distinguish between two cases.
\begin{enumerate}
\item $14\in A$. We have
$R_5-14^5=87488$. But $10^5>87488$ and therefore $10,11,12,13\notin A$. Since $\sum_{k=1}^{8}k^5=61776<87488$, necessarily $9\in A$. We have $R_5-14^5-9^5=28439$. But $8^5>28439$ and therefore $8\notin A$. Since $\sum_{k=1}^{6}k^5=12201<28439$, necessarily $7\in A$. We have $R_5-14^5-9^5-7^5=11632$. Since $\sum_{k=1}^{5}k^5=4425<11632$, necessarily $6\in A$. Then $R_5-14^5-9^5-7^5-6^5=3856$. Since $\sum_{k=1}^{4}k^5=1300<3856$, necessarily $5\in A$. We have $R_5-14^5-9^5-7^5-6^5-5^5=731$. But $4^5>731$ and therefore $4\notin A$. On the other hand, $\sum_{k=1}^{3}k^5=276<731$, concluding the proof that $14\in A$ is impossible.

\item $14\notin A$. Again, we distinguish between two cases.
\begin{enumerate}
\item $13\in A$. We have $R_5-13^5=254019$. Assume that $12\in A$. We have $R_5-13^5-12^5=5187$. But $6^5>5187$ and therefore $6,7,8,9,10,11\notin A$. On the other hand $\sum_{k=1}^{5}k^5=4425<5187$, showing that $12\notin A$. Since $\sum_{k=1}^{10}k^5=220825<254019$, necessarily $11\in A$. We have $R_5-13^5-11^5=92968$. But $10^5>92968$ and therefore $10\notin A$. Since $\sum_{k=1}^{8}k^5=61776<92968$, necessarily $9\in A$. We have $R_5-13^5-11^5-9^5=33919$. Since
$\sum_{k=1}^{7}k^5=29008<33919$, necessarily $8\in A$. We have $R_5-13^5-11^5-9^5-8^5=1151$. But $5^5>1151$ and therefore $5,6,7\notin A$. Since $\sum_{k=1}^{3}k^5 = 276<1151$, necessarily $4\in A$. We have $R_5-13^5-11^5-9^5-8^5-4^5=127$. But $3^5=243>127$ and therefore $3\notin A$. On the other hand, $\sum_{k=1}^{2}k^5=33<127$, showing that $13\in A$ is impossible.

\item $13\notin A$. Since $\sum_{k=1}^{11}k^5=381876<625312$, necessarily $12\in A$. We have $R_5-12^5=376480$. Since $\sum_{k=1}^{10}k^5=220825<376480$, necessarily $11\in A$. We have $R_5-12^5-11^5=215429$. Since
$\sum_{k=1}^{9}k^5=120825<215429$, necessarily $10\in A$. We have $R_5-12^5-11^5-10^5=115429$. Since $\sum_{k=1}^{8}k^5=61776<115429$, necessarily $9\in A$. We have $R_5-12^5-11^5-10^5-9^5=56380$. Since $\sum_{k=1}^{7}k^5=29008<56380$, necessarily $8\in A$. We have $R_5-12^5-11^5-10^5-9^5-8^5=23612$. Since $\sum_{k=1}^{6}k^5=12201<23612$, necessarily $7\in A$. We have $R_5-12^5-11^5-10^5-9^5-8^5-7^5=6805$. But $6^5=7776>6805$ and therefore $6\notin A$. Since $\sum_{k=1}^{4}k^5=1300<6805$, necessarily $5\in A$. We have $R_5-12^5-11^5-10^5-9^5-8^5-7^5-5^5=3680$. Since $\sum_{k=1}^{3}k^5=276<3680$, necessarily $4\in A$. We have $R_5-12^5-11^5-10^5-9^5-8^5-7^5-5^5-4^5=2656$. But $\sum_{k=1}^{3}k^5=276<2656$, showing that also $13\notin A$ is impossible. This concludes the proof that $14\notin A$ is impossible, which, in turn, concludes the proof that $n=5$ is impossible. \qedhere
\end{enumerate}
\end{enumerate}
\end{enumerate}
\end{proof}

\subsection{\texorpdfstring{\seqnum{A222945}}{}}

For $n\in\NN$ let
\[
S_n=\{i+j+k\;:\; i,j,k\in\ZZ, |i|, |j|, |k|, |ijk|\leq n\}.
\]
Mathar conjectured in \seqnum{A222945} that $\#S_n = 4n+1$. The statement of the following theorem confirms the conjecture.

\begin{theorem}
Let $n\in\NN$ with $n\geq 2$. Then
$S_n= \{t\in\ZZ\;:\; |t|\leq 2n\}$.
\end{theorem}

\begin{proof}
``$\{t\in\ZZ\;:\; |t|\leq 2n\}\subseteq S_n"$. Let $t\in\ZZ$ with $|t|\leq 2n$. If $|t|\leq n$, then
$t=t+0+0$ and $|t|, |0|, |t\cdot 0\cdot 0| \leq n$. Thus, $t\in S_n$. Now assume that $|t|>n$. Write
$t=\varepsilon(n+m)$ with
$\varepsilon\in\{-1,1\}$ and $m=|t|-n$. Then $t=\varepsilon n+\varepsilon m+0$ and $|\varepsilon n|, |\varepsilon m|, |0|, |\varepsilon n\cdot \varepsilon m\cdot 0|\leq n$. Thus, $t\in S_n$.

``$S_n\subseteq \{t\in\ZZ\;:\; |t|\leq 2n\}"$. Let $s\in S_n$ and let $i,j,k\in\ZZ$ with $s=i+j+k$ and $|i|,|j|,|k|,|ijk|\leq n$. Assume that at least one of $i,j,k$ is zero. Without loss, $i=0$. Then $|s|=|0+j+k|\leq  |j|+|k|\leq 2n$. Now assume that $i,j,k\ne 0$. We claim that $|i|+|j|+|k|\leq |ijk|+2$. Indeed, 
\begin{align*}
&|ijk|-|i|-|j|-|k|+2
=\\
&(|i|-1)(|j|-1)(|k|-1)+(|i|-1)(|j|-1)+(|j|-1)(|k|-1)+(|k|-1)(|i|-1).
\end{align*}
Since $|i|,|j|,|k|\ge 1$, the right-hand side is nonnegative. Moreover, since $n\ge 2$, we have $n+2\leq 2n$. It follows that
\[
|s|=|i+j+k|\leq |i|+|j|+|k|\leq |ijk|+2\leq n+2\leq 2n.\qedhere
\] 
\end{proof}

\subsection{\texorpdfstring{\seqnum{A242932}}{}}

Sequence \seqnum{A242932} consists of the numbers $n\in\NN$ for which there exist $k\in\NN$ and a prime number $p$ such that
\begin{equation}\label{orr1}
\frac{kn}{k+n}=p.
\end{equation} Orr conjectured in \seqnum{A242932} that the sequence contains no squares of even integers, other than $4$. In the following theorem we provide a characterization of the sequence. The subsequent corollary confirms the conjecture.

\begin{theorem}
Let $n\in\NN$. Then $n\in \seqnum{A242932}$ if and only if there exists a prime number $p$ such that
$n\in\{p+1,2p,p(p+1)\}$.
\end{theorem}

\begin{proof}
Let $n\in\NN$ and assume that \eqref{orr1} holds for $n$ for some $k\in\NN$ and some prime number $p$.
Rewriting \eqref{orr1} we obtain
$kn-pk-pn=0$. Adding $p^2$ to both sides yields $(n-p)(k-p)=p^2$. Since $\frac{k}{k+n}<1$, it follows from \eqref{orr1} that $p<n$ and therefore $n-p>0$. Similarly, since $\frac{n}{k+n}<1$, we have $p<k$ and therefore $k-p>0$. Since $p$ is prime, $n-p\in\{1,p,p^2\}$. Equivalently, $n\in\{p+1,2p,p(p+1)\}$. 

Conversely, let $n\in\NN$ and assume that there exists a prime number $p$ such that $n\in\{p+1,2p,p(p+1)\}$. We distinguish between the three possibilities.
\begin{enumerate}
\item $n=p+1$. Set $k=p^2+p$. We have
\[
\frac{kn}{k+n}
=\frac{(p^2+p)(p+1)}{(p^2+p)+(p+1)}
=\frac{p(p+1)^2}{(p+1)^2}
=p.
\] 
\item $n=2p$. Set $k=2p$. We have
\[
\frac{kn}{k+n}
=\frac{(2p)(2p)}{2p+2p}
=\frac{4p^2}{4p}
=p.
\]
\item $n=p(p+1)$. Set $k=p+1$. We have
\[
\frac{kn}{k+n}
=\frac{(p+1)p(p+1)}{(p+1)+p(p+1)}
=\frac{p(p+1)^2}{(p+1)^2}
=p.
\]
\end{enumerate}
In each of the three cases, \eqref{orr1} holds for $n$.
\end{proof}

\begin{corollary}
Let $m\in\NN$ with $m\geq 2$. Then $4m^2\notin \seqnum{A242932}$.
\end{corollary}

\begin{proof}
By the previous theorem, if $4m^2\in\seqnum{A242932}$, then $4m^2\in\{p+1,2p,p(p+1)\}$ for some prime $p$. Assume that $4m^2=p+1$. Then
\[
p=4m^2-1=(2m-1)(2m+1),
\]
which is composite for $m\ge 2$, a contradiction. Assume now that $4m^2=2p$. Then $p=2m^2$, which is even and greater than $2$ and therefore not prime. A contradiction. Finally, since $p^2 <p(p+1)<(p+1)^2$, the number $p(p+1)$ can not be a perfect square, in particular, it can not be $4m^2$.
\end{proof}

\subsection{\texorpdfstring{\seqnum{A242933}}{}}

Sequence \seqnum{A242933} consists of all even numbers $n\in\NN$ such that for every $k\in\NN$ the number $\frac{nk}{n+k}$ is not prime. The statement of the following theorem was conjectured by Orr in \seqnum{A242933}.

\begin{theorem}
Let $n\in\NN$ with $n\ge 2$. Then, for every $k\in\NN$, the number
$\frac{4n^2k}{4n^2+k}$ is not prime.
\end{theorem}

\begin{proof}
Assume that for some $k\in\NN$, the number
$\frac{4n^2k}{4n^2+k}$ is a prime number $p$. Let $g=\gcd(k,4n^2)$ and write $k=gu, 4n^2=gv$ for some $u,v\in\NN$.
Then $p=\frac{guv}{u+v}$. Since $\gcd(u,v)=1$, we have
\[
\gcd(u+v,u)=\gcd(u+v,v)=1.
\]
Thus, $\gcd(u+v,uv)=1$ and therefore $(u+v)\mid g$. Hence, we may write $g=t(u+v)$ for some $t\in\NN$. It follows that $p=tuv$. Since $p$ is prime, necessarily, exactly one of the three numbers $t,u,v$ is equal to $p$, and the other two are equal to $1$. We distinguish between the three cases.
\begin{enumerate}
\item $t=p$ and $u=v=1$. Then
$4n^2=p\cdot 1\cdot (1+1)=2p$. Thus, $p=2n^2$, a composite since $n\ge 2$. 
\item $u=p$ and $t=v=1$. Then
$4n^2=1\cdot 1\cdot (p+1)=p+1$. Thus,
$p=4n^2-1=(2n-1)(2n+1)$, a composite since $n\ge 2$. 
\item $v=p$ and $t=u=1$. Then
$4n^2=1\cdot p\cdot (1+p)=p(p+1)$. If $p=2$, then $4n^2=6$, which is impossible. Thus, $p>2$ and therefore $p\mid n$. Write $n=pr$ for some $r\in\NN$. Then
$p(p+1)=4p^2r^2$ and therefore
$p+1=4pr^2$, which is impossible.
\end{enumerate}
In all cases we obtain a contradiction. Therefore, for every $k\in\NN$, the number
$\frac{4n^2k}{4n^2+k}$ is not prime.
\end{proof}

\subsection{\texorpdfstring{\seqnum{A242992}}{}}

For $n\in\NN_0$ let $a_{n}$ be the smallest $k\in\NN$ such that
\begin{equation}\label{e1992}
\frac{n}{2}<k<n
\qquad\text{and}\qquad
(2^{n-k}-1) \mid (2^k-2),
\end{equation}
provided that such $k$ exists. Otherwise, set $a_{n}=0$. For $n\in\NN$ with $n\geq 2$ let $b_{n}$ denote the largest proper divisor of $n$, i.e., $b_{n}=\frac{n}{p}$ where $p$ is the least prime factor of $n$. The statement of Theorem \ref{OUDRA} below was conjectured by Oudra in \seqnum{A242992}. We shall need the following standard result (e.g., \cite[Corollary 3.36]{ADM}).

\begin{lemma}\label{lemADM}
Let $a,b\in\NN$ with $a>b$ and $\gcd(a,b)=1$. Then, for every $m,n\in\NN$, we have
\[(a^m - b^m)\mid (a^n-b^n) \iff m\mid n.\]
\end{lemma}

\begin{theorem}\label{OUDRA}
Let $n\in\NN$ with $n\geq 3$. Then $a_{n}=n-b_{n-1}$.
\end{theorem}

\begin{proof}
Let $k\in\NN$ such that $\frac{n}{2}<k<n$. Clearly, this condition is equivalent to $1\leq n-k<\frac{n}{2}$. Since $n\geq 3$, we have $k\geq 2$. Furthermore, since $2^{n-k}-1$ is odd, we have
\begin{align*}
(2^{n-k}-1) \mid (2^k-2)&\iff (2^{n-k}-1) \mid (2^{k-1}-1)\\
&\iff (n-k) \mid (k-1)\\
&\iff (n-k) \mid (n-1),
\end{align*} where in the second equivalence we used Lemma \ref{lemADM}. It follows that $k$ satisfies \eqref{e1992} if and only if
\[
1 \leq n-k<\frac{n}{2} \qquad\text{and}\qquad (n-k) \mid (n-1).
\]
Setting $t=n-k$, we see that the smallest admissible value of $k$ corresponds to the largest admissible value of $t$. Hence,
\begin{equation}
a_{n}=n-\max\left\{1\leq t<\frac n2\;:\; t\mid (n-1) \right\}.
\end{equation}
Thus, it suffices to prove that
\begin{equation}\label{e3992}
\max\left\{1\leq t<\frac n2\;:\; t\mid (n-1) \right\}=b_{n-1}.
\end{equation}
Since $n\geq 3$, we have $\frac{n}{2}<n-1$ and therefore the condition $t<\frac n2$ excludes $t=n-1$. First, assume that $n$ is odd. Then $n-1$ is even and therefore $b_{n-1}=\frac{n-1}{2}<\frac n2$. Thus, \eqref{e3992} holds. Now assume that $n$ is even. Then $n-1$ is odd. Suppose that $n-1$ is prime. Then $b_{n-1}=1$ and no $1\leq t < \frac{n}{2}<n-1$ divides $n-1$, except $1$. Thus, \eqref{e3992} holds. Finally, assume that $n-1$ is composite and let $p$ be its least prime factor. Necessarily, $p\ge 3$, and therefore $1\leq b_{n-1}=\frac{n-1}{p}\leq \frac{n-1}{3}< \frac{n}{2}$. Thus, the left-hand side of \eqref{e3992} is $\geq b_{n-1}$. On the other hand, $b_{n-1}$ is the largest proper divisor of $n-1$ and the maximum on the left-hand side of \eqref{e3992} is taken over divisors of $n-1$, which are less than $\frac{n}{2}<n-1$. Thus, the left-hand side of \eqref{e3992} is $\leq b_{n-1}$. Thus, \eqref{e3992} holds.
\end{proof}

\subsection{\texorpdfstring{\seqnum{A328564}, \seqnum{A328565}, and \seqnum{A328566}}{}}

For $n\in\NN_0$ define the three sets
\begin{align*}
A_n&=\{(n-k)\land k\;:\; k=0,1,\dots,n\},\\
X_n&=\{(n-k)\oplus k\;:\; k=0,1,\dots,n\},\\
O_n&=\{(n-k)\lor k\;:\; k=0,1,\dots,n\}.
\end{align*}
Let $(s_n)_{n\in\NN_0}$ denote Stern's diatomic sequence (cf.\ \seqnum{A002487}), defined as follows: $s_0=0, s_1=1$ and, for every $m\in\NN_0$, 
\[s_{2m}=s_m,\qquad s_{2m+1}=s_m+s_{m+1}.
\]
The statements of Theorems \ref{t1564} and \ref{sig23} were conjectured by Sigrist in \seqnum{A328564}, \seqnum{A328565}, and \seqnum{A328566}. 

\begin{lemma}\label{l1564}
Let $m\in\NN_0$. Then
\begin{align}
A_{2m+1}&=\{2t\;:\; t\in A_m\},\label{e1564}\\
X_{2m+1}&=\{2t+1\;:\; t\in X_m\},\label{e2564}\\
O_{2m+1}&=\{2t+1\;:\; t\in O_m\}.\label{e3}
\end{align}
Furthermore, if $m\ge 1$ then
\begin{align}
A_{2m}&=\{2t\;:\; t\in A_m\} \dot\cup \{2t+1\;:\; t\in A_{m-1}\},\label{e4564}\\
X_{2m}&=\{2t\;:\; t\in X_m\} \dot\cup \{2t\;:\; t\in X_{m-1}\},\label{e5564}\\
O_{2m}&=\{2t\;:\; t\in O_m\}\dot\cup \{2t+1\;:\; t\in O_{m-1}\}.\label{e6}
\end{align}
\end{lemma}

\begin{proof}
The key observation is that, for every
$\star\in\{\land,\lor,\oplus\}$, $a,b\in\NN_0$, and
$\varepsilon,\delta\in\{0,1\}$,
\[
(2a+\varepsilon)\star(2b+\delta)
=
2(a\star b)+(\varepsilon\star\delta).
\]
Let $j\in\{0,1,\ldots,m\}$. Then
\begin{align*}
(2m+1-2j)\star(2j)
&=(2m+1-(2j+1))\star(2j+1)\\
&=
\begin{cases}
2((m-j)\star j),
&\textnormal{if $\star=\land$},\\
2((m-j)\star j)+1,
&\textnormal{if $\star\in\{\lor,\oplus\}$}.
\end{cases}
\end{align*}
Since every $k\in\{0,1,\ldots,2m+1\}$ can be written uniquely in the form $2j$ or $2j+1$, this proves \eqref{e1564}, \eqref{e2564}, and \eqref{e3}. Assume now that $m\geq1$. For every $j\in\{0,1,\ldots,m\}$ and each $\star\in\{\land,\lor,\oplus\}$ we have $(2m-2j)\star(2j)=2((m-j)\star j)$. For every $j\in\{0,1,\ldots,m-1\}$ we have
\[
(2m-(2j+1))\star(2j+1)
=
\begin{cases}
2((m-1-j)\star j),
&\textnormal{if $\star=\oplus$},\\
2((m-1-j)\star j)+1,
&\textnormal{if $\star\in\{\land,\lor\}$}.
\end{cases}
\]
These identities give the unions in \eqref{e4564}, \eqref{e5564}, and \eqref{e6}. It remains to prove that the unions are disjoint. The disjointness in \eqref{e4564} and \eqref{e6} is clear, since one set consists of even integers and the other of odd integers. For the disjointness in \eqref{e5564}, notice that for every $r\in\NN_0$ and $y\in X_r$, we have $y\equiv r \pmod 2$. Indeed, let $k\in\{0,1,\ldots,r\}$ such that $y=(r-k)\oplus k$. Since the least significant bit of the XOR of two integers is the XOR of their least significant bits, and XOR of two bits agrees with their sum modulo $2$, we have \[y\equiv (r-k)\oplus k \equiv (r-k)+k \equiv r \pmod 2. \] Thus, every element of $X_r$ has the same parity as $r$. Hence, for every $m\in\NN$, we have $X_m\cap X_{m-1}=\emptyset$, since $m$ and $m-1$ have opposite parities.
\end{proof}

\begin{theorem}\label{t1564}
Let $n\in\NN_0$. Then $\# A_n= \# X_n=\# O_n =s_{n+1}$.
\end{theorem}

\begin{proof}
By \eqref{e1564}, for every $m\in\NN_0$ we have
$\# A_{2m+1}=\# A_m$ and, if $m\geq 1$, by \eqref{e4564}, $\# A_{2m}=\# A_m+\# A_{m-1}$. Furthermore, $\# A_0=\# A_1=1$. Now, for $m\in\NN_0$ set $d_m=s_{m+1}$. Then $d_0=d_1=1$.
Furthermore, $d_{2m+1}=s_{2m+2}=s_{m+1}=d_m$, and, if $m\ge 1$,
$d_{2m}=s_{2m+1}=s_m+s_{m+1}=d_{m-1}+d_m$. Thus, the sequences $(\# A_n)_{n\in\NN_0}$ and $(d_n)_{n\in\NN_0}$ satisfy the same recurrences and have the same initial values. Hence, $\# A_n=d_n=s_{n+1}$ for every $n\in\NN_0$. By \eqref{e2564} and \eqref{e5564}, the sequence $(\# X_n)_{n\in\NN_0}$ satisfies the same recurrences, and $\# X_0=\# X_1=1$. Similarly, by \eqref{e3} and \eqref{e6}, the sequence $(\# O_n)_{n\in\NN_0}$ satisfies the same recurrences, and $\# O_0=\# O_1=1$. Therefore, $\# X_n=\# O_n=s_{n+1}$, for every $n\in\NN_0$.
\end{proof}

\begin{theorem}\label{sig23}
Let $n\in\NN_0$ and let 
\[
a_n=\sum_{t\in A_n} t,\qquad
x_n=\sum_{t\in X_n} t,\qquad 
o_n=\sum_{t\in O_n} t.
\]
Then $a_n+x_n=o_n$.
\end{theorem}

\begin{proof}
Set $c_n=\# A_n $. By Theorem \ref{t1564}, we have $\# X_n =\# O_n =c_n$. Let $m\in\NN_0$. By \eqref{e1564}, \eqref{e2564}, and
\eqref{e3},
\begin{align}
a_{2m+1}&=\sum_{t\in A_m} 2t = 2a_{m},\label{xx1}\\
x_{2m+1}&=\sum_{t\in X_m} (2t+1)=2x_{m}+\# X_m =2x_{m}+c_{m},\label{xx2}\\
o_{2m+1}&=\sum_{t\in O_m} (2t+1)=2o_{m}+\# O_m =2o_{m}+c_{m}.\label{xx3}
\end{align}
By \eqref{e4564}, \eqref{e5564}, and \eqref{e6}, if $m\ge 1$, then
\begin{align}
a_{2m}&=\sum_{t\in A_m} 2t+ \sum_{t\in A_{m-1}} (2t+1)=2a_{m}+2a_{m-1}+c_{m-1},\label{xc1}\\
x_{2m}&=\sum_{t\in X_m}2t+ \sum_{t\in X_{m-1}}2t=2x_{m}+2x_{m-1},\label{xc2}\\
o_{2m}&=\sum_{t\in O_m}2t+ \sum_{t\in O_{m-1}}(2t+1)=2o_{m}+2o_{m-1}+c_{m-1}.\label{xc3}
\end{align} For $n\in\NN_0$ set $D_{n}=o_n-a_n-x_n$. We prove that $D_n=0$, for every $n\in\NN_0$. By \eqref{xx1}, \eqref{xx2}, and \eqref{xx3}, for every $m\in\NN_0$, 
\[
D_{2m+1}=o_{2m+1}-a_{2m+1}-x_{2m+1}=2o_{m}+c_{m} -2a_{m}-2x_{m}-c_{m}=2D_{m}.
\]
Furthermore, by \eqref{xc1}, \eqref{xc2}, and \eqref{xc3}, for every $m\ge 1$ we have
\begin{align*}
D_{2m}&=o_{2m}-a_{2m}-x_{2m}\\
&=2o_{m}+2o_{m-1}+c_{m-1}-2a_{m}-2a_{m-1}-c_{m-1}-2x_{m}-2x_{m-1}\\
&=2(D_{m}+D_{m-1}).
\end{align*}
We have $A_0=A_1=O_0=X_0=\{0\}$ and $X_1 = O_1 =\{1\}$. Thus,
$D_{0}=D_{1}=0$. It follows that $D_{n}=0$ for every $n\in\NN_0$.
\end{proof}

\subsection{\texorpdfstring{\seqnum{A176542}}{}}

For $j\in \NN_0$ and $n\in\NN$, let 
\begin{align*}
S_n(j)&=\sum_{k=j}^{j+n-1}T_k,\\
N(n)&=\#\{j\in\NN_0\;:\;S_n(j) \textnormal{ is a square}\}.
\end{align*} Sequence \seqnum{A176542} consists of the numbers $n\in\NN$ such that $0<N(n)<\infty$.
The statements of Theorem \ref{th42} and of Corollary \ref{c142} were conjectured by Barker in \seqnum{A176542}. We shall need the following result.

\begin{lemma}\label{ll142}
Let $C\in\ZZ$ and let $D\in\NN$ such that $8\mid D$ and $D$ is not a square. Assume that the equation
\begin{equation}\label{eb142}
X^2-DY^2=C
\end{equation}
has a solution $(X_0,Y_0)\in\NN^2$ with $8\mid X_0$. Then \eqref{eb142} has infinitely many solutions $(X_k,Y_k)\in\NN^2$ such that for every $k\in\NN_0$, $8\mid X_k$ and $Y_k\equiv Y_0\pmod 2$. Furthermore, $\lim_{k\to\infty}Y_k=\infty$.
\end{lemma}

\begin{proof}
It is well-known (e.g., \cite[Section~4.2]{Barbeau}) that the equation
\begin{equation}\label{p1}
u^2-Dv^2=1
\end{equation}
has a solution $(u_1,v_1)\in\NN^2$. Set $u=u_1^2+Dv_1^2$ and $v= 2u_1v_1$. Then $(u,v)$ solves \eqref{p1}, $u$ is odd, and $v$ is even. Let $\alpha=u+v\sqrt D$ and $\beta=u-v\sqrt D$. Then $\alpha>1$ and, since $\alpha\beta=u^2-Dv^2=1$, we have $\beta=\alpha^{-1}$. Thus, $0<\beta<1$. For $k\in\NN_0$ let $(X_k,Y_k)\in\NN^2$ be defined by
\begin{equation}\label{ef142}
X_k+Y_k\sqrt D=(X_0+Y_0\sqrt D)\alpha^k.
\end{equation}
Then 
\begin{equation}\label{fl242}
X_k-Y_k\sqrt D=(X_0-Y_0\sqrt D)\beta^k.
\end{equation}
Subtracting \eqref{fl242} from \eqref{ef142} and solving for $Y_k$ yields
\[
Y_k=\frac{1}{2\sqrt D}((X_0+Y_0\sqrt D)\alpha^k-(X_0-Y_0\sqrt D)\beta^k).
\]
Since $X_0,Y_0>0$, we have $X_0+Y_0\sqrt D>0$. Since $\alpha^k\underset{k\to\infty}{\longrightarrow}\infty$ and $\beta^k\underset{k\to\infty}{\longrightarrow} 0$, we have $Y_k\underset{k\to\infty}{\longrightarrow}\infty$. Multiplying \eqref{ef142} with \eqref{fl242} gives
\[
X_k^2-DY_k^2=(X_0^2-DY_0^2)(u^2-Dv^2)^k=C\cdot 1^k=C.
\] Thus, $(X_k,Y_k)$ solves \eqref{eb142}.
We have
\[
X_{k+1}+Y_{k+1}\sqrt D=(X_k+Y_k\sqrt D)(u+v\sqrt D).
\]
Thus, 
$X_{k+1}=uX_k+DvY_k$ and $Y_{k+1}=vX_k+uY_k$. Thus, inductively, since $v$ is even and $u$ is odd, $Y_{k+1}\equiv Y_k \equiv Y_0 \pmod 2$. Moreover, since $8\mid D$ and, inductively, $8\mid X_k$, we have $8\mid X_{k+1}$. Since $Y_k\to\infty$, the solutions $(X_k,Y_k)$ are infinitely many
and distinct.
\end{proof}

\begin{theorem}\label{th42}
Let $n\in\NN$. Then $0<N(n)<\infty$ if and only if $n=2m^2$ for some $m\in\NN$ with $m\geq 4$ and $3\nmid m$.
\end{theorem}

\begin{proof}
For $j\in\NN_0$, we have
\begin{equation}
S_n(j)=\sum_{k=j}^{j+n-1}\frac{k(k+1)}{2}=\frac{n(3j^2+3jn+n^2-1)}{6}.\label{esn42}
\end{equation} Suppose that $S_n(j)=x^2$ for some $x\in\NN_0$ and set $y=2j+n$.
Then $y\geq n$ and $y\equiv n\pmod 2$. Moreover, using
\eqref{esn42}, we obtain
\begin{align}
8x^2-ny^2&=8S_n(j)-n(2j+n)^2\nonumber\\
&=\frac{4n}{3}(3j^2+3jn+n^2-1)-n(2j+n)^2\nonumber\\
&=\frac{n(n^2-4)}{3}.\label{ew142}
\end{align}
Conversely, suppose that $x,y\in\NN_0$ satisfy \eqref{ew142} with
$y\geq n$ and $y\equiv n\pmod 2$. Set $j=\frac{y-n}{2}$. Then $j\in\NN_0$ and reversing the calculation above gives $S_n(j)=x^2$. It follows that
\[
N(n)=
\#\left\{
(x,y)\in\NN_0^2:
\begin{array}{l}
y\geq n,\quad y\equiv n\pmod 2,\\[1mm]
\displaystyle
8x^2-ny^2=\frac{n(n^2-4)}{3}
\end{array}
\right\}.
\]

Assume that $0<N(n)<\infty$ and suppose that $8n$ is not a square. If $n=1$, then $S_1(1)=1$. If $n\geq2$, then $S_n(j)>0$ for every $j\in\NN_0$. Thus, in either case, there exist $j\in\NN_0$ and $x\in\NN$ such that $S_n(j)=x^2$. Then $x$ and $y=2j+n$ solve \eqref{ew142}. Set $X_0=8x, Y_0=y, C=\frac{8n(n^2-4)}{3}$, and $D=8n$. Then $(X_0,Y_0)$ solves the equation $X^2-DY^2=C$. By Lemma~\ref{ll142}, the equation has infinitely many solutions
$(X_k,Y_k)$ such that
\[
\lim_{k\to\infty}Y_k=\infty,\qquad
Y_k\equiv Y_0\equiv n\pmod 2,
\qquad
8\mid X_k.
\]
Hence, for all sufficiently large $k$, we may set $j_k=\frac{Y_k-n}{2}$ and $x_k=\frac{X_k}{8}$. Reversing the preceding calculation gives $S_n(j_k)=x_k^2$. Since $Y_k\to\infty$, the integers $j_k$ take infinitely many distinct values. Thus, $N(n)=\infty$, contradicting
$N(n)<\infty$. Thus, $8n$ is a square and therefore $n=2m^2$, for some $m\in\NN$. With this \eqref{ew142} may be written as
\begin{equation}\label{ma142}
(2x-my)(2x+my)=\frac{4m^2(m^4-1)}{3}.
\end{equation}
Assume that $m=1$, i.e., $n=2$. Then $S_2(j)=(j+1)^2$, for every $j\in\NN_0$. Thus, $N(2)=\infty$ and therefore $m=1$ is impossible. Now assume that $m=2$, i.e., $n=8$. Then \eqref{ma142} becomes
$(x-y)(x+y)=20$. Since $x-y$ and $x+y$ have the same parity, they must both be even. Write $x-y=2u$ and $x+y=2v$, for some $u,v\in\NN$. Then $uv=5$ allowing only $(u,v)=(1,5)$ or $(u,v)=(5,1)$. The first case yields $y=4$ and the second $y=-4$, both are impossible since we need $y\geq 8$. Consequently $N(8)=0$ and therefore $m=2$ is impossible. Now assume that $m\geq 3$. In particular, the right-hand side of \eqref{ma142} is nonzero and it does not depend on $x$ and $y$. Thus, equation \eqref{ma142} has only finitely many solutions $(x,y)\in\NN_0^2$. We conclude that, for $n=2m^2$ with $m\geq 2$, we have $N(n)<\infty$. 

We now prove that if $3\mid m$ then $N(n)=0$. To this end, let $r=v_3(m)$. By assumption, $r\geq 1$. Suppose that $(x,y)\in\NN_0^2$ is a solution of \eqref{ma142}. Set $A=2x-my$ and $B=2x+my$. Then \[AB=\frac{4m^2(m^4-1)}{3}, \qquad B-A=2my.\] Thus, $3^r\mid (B-A)$ and therefore $A\equiv B \pmod{3^r}$. Since $3\nmid (m^4-1)$, we have
\begin{equation}\label{oddnumber}
v_3(AB)=v_3\left(\frac{4m^2(m^4-1)}{3}\right)=2r-1.
\end{equation} Now, if $v_3(A)<r$, then $v_3(B)=v_3(A)$ and therefore, $v_3(AB)=2v_3(A)$, an even number, contradicting \eqref{oddnumber}. On the other hand, if $v_3(A)\geq r$ then $v_3(B)\geq r$. Thus, $v_3(AB)\ge 2r>2r-1$, contradicting \eqref{oddnumber}. This concludes the proof that if $3\mid m$ then $N(n)=0$ and the proof of the ``only if" part of the statement is complete.

Conversely, assume that $n=2m^2$ with $m\in\NN$ such that $m\geq 4$ and $3\nmid m$. Set 
\[(x,y) = 
\begin{cases}
\left(\frac{m(m^4+2)}{6},\frac{m^4-4}{3}\right), &\textnormal{if $m$ is even},\\
\left(\frac{m(m^4+11)}{12},\frac{m^4-13}{6}\right), &\textnormal{if $m$ is odd}.
\end{cases}\]
In either case, $x,y\in\NN$. Furthermore, $y$ is even and therefore $y\equiv n\pmod 2$. In addition, since $m\geq 4$, we have $y\geq n$. Finally, direct calculations show that $x$ and $y$ satisfy equation \eqref{ew142}. Therefore, by the previous arguments, there exists $j\in\NN_0$ such that $S_n(j)=x^2$. Hence, $N(n)>0$. Since we have already proved that $N(n)<\infty$ for $n=2m^2$ with $m\geq2$, it follows that $0<N(n)<\infty$.
\end{proof}

\begin{corollary}\label{c142}
For $n\in\NN$ let $a(n)$ be the $n$th smallest $m\in\NN$ for which $0<N(m)<\infty$. Then
\[
a(n)=
\begin{cases}
\frac{9n^2+24n+16}{2}, & \textnormal{if $n$ is even},\\
\frac{9n^2+30n+25}{2}, & \textnormal{if $n$ is odd}.
\end{cases}
\]
\end{corollary}

\begin{proof}
By Theorem \ref{th42}, 
\[
\{m\in\NN\;:\;0<N(m)<\infty\}=\{2m^2\;:\; m\in\NN, m\geq 4, 3\nmid m\}.
\] Thus, it suffices to enumerate in increasing order the set
\[
\mathcal{M}=\{m\in\NN\;:\; m\geq 4, 3\nmid m\}
\]
and then apply the map $x\mapsto 2x^2$ to its elements. To this end, define a sequence $(m_n)_{n\in\NN}$ by $m_{2n}=3n+2$ and $m_{2n-1}=3n+1$, for $n\in\NN$. Clearly, $m_n\in\mathcal{M}$, for every $n\in\NN$. Furthermore, $(m_n)_{n\in\NN}$ is strictly increasing. Indeed, for every $n\in\NN$, we have
\[
m_{2n-1}=3n+1 < 3n+2=m_{2n} < 3n+4=m_{2n+1}.
\]
Conversely, let $m\in\mathcal{M}$. If $m\equiv 1\pmod 3$ then $m=3q+1$ for some $q\in\NN$. Thus, $m=m_{2q-1}$. If $m\equiv 2\pmod 3$ then $m=3q+2$ for some $q\in\NN$. Thus, $m=m_{2q}$. This concludes the proof that $a(n)=2m_n^2$, for every $n\in\NN$. Finally, let $q\in\NN$ and assume that $n=2q$. Then
\[
a(n)=2(3q+2)^2=2\left(\frac{3n}{2}+2\right)^2=\frac{9n^2+24n+16}{2}.
\]
Similarly, if $n=2q-1$, then
\[
a(n)=2(3q+1)^2
=2\left(\frac{3(n+1)}{2}+1\right)^2=\frac{9n^2+30n+25}{2}.\qedhere
\]
\end{proof}

\subsection{\texorpdfstring{\seqnum{A181176}}{}}

For $n\in\NN_0$ let $a_n$ denote the $n$th element of \seqnum{A181176}, which is defined by
\[a_n=\min_{k\in\{0,1,\ldots,n\}} \left|\sum_{j=1}^k j-\sum_{j=k+1}^n j\right|.\] 
Clearly, 
\[
a_n=\min_{k\in\{0,1,\ldots,n\}} |T_k-(T_n-T_k)|=\min_{k\in\{0,1,\ldots,n\}} |T_n-2T_k|.
\]
In the comments to \seqnum{A181176} Wilson v observes that certain numbers seem not to occur as values in the sequence. In this section we prove that this is indeed the case and identify the arithmetic reason for this. We begin by providing an almost closed form formula for $a_n$.

\begin{theorem}
Let $n\in\NN_0$ and let 
\[
k_n=\left\lfloor\frac{-1+\sqrt{1+2n(n+1)}}{2}\right\rfloor. 
\] Then 
\[
a_n=\min\{T_n-2T_{k_n}, 2T_{k_n+1}-T_n\}.
\]
\end{theorem}

\begin{proof}
Let $n,k\in\NN_0$ with $k\leq n$ and set $S_n(k)=T_n-2T_k$. For $n=0$ the statement holds trivially. Thus, assume that $n\geq 1$. We claim that the map $k\mapsto S_n(k)$ is
strictly decreasing on $\{0,1,\ldots,n\}$. Indeed, for $k\in \{0,1,\ldots,n-1\}$ we have
\begin{align*}
S_n(k+1)-S_n(k)
&=T_n-2T_{k+1}-(T_n-2T_k)\\
&=-2(T_{k+1}-T_k)\\
&=-2(k+1)<0.
\end{align*} Let
\begin{align*}
k_n&=\max\{k\in\{0,1,\dots,n\}\;:\;T_n-2T_k\geq 0\}\\
&=\max\left\{k\in\{0,1,\dots,n\}\;:k^2+k-\frac{n(n+1)}{2}\leq 0\right\}.
\end{align*}
Thus,
\[k_n =\left\lfloor\frac{-1+\sqrt{1+2n(n+1)}}{2}\right\rfloor. \] Notice that since $T_n-2T_n<0$, necessarily $k_n <n$.
Since $k\mapsto S_n(k)$ is strictly decreasing and since $S_n(k_n)\geq 0$ and $S_n(k_n+1) <0$, the minimum value of $|T_n-2T_k|$ is attained either at $k_n$ or at $k_n+1$, and the proof is complete.
\end{proof}

\begin{lemma}\label{l176}
Let $d\in\NN_0$. Then there exist $n,k\in\NN_0$ with $|T_n-2T_k|=d$
if and only if there exist odd $x,y\in\NN$ with
\begin{equation}\label{e3376}
x^2-2y^2\in\{8d-1,-8d-1\}.
\end{equation}
\end{lemma}

\begin{proof}
Assume that there exist $n,k\in\NN_0$ with $|T_n-2T_k|=d$. Let $x=2n+1$ and $y=2k+1$. Then $x,y\in\NN$ and are both odd. Furthermore,
\begin{equation}\label{e9176}
T_n=\frac{x^2-1}{8},\qquad 2T_k=\frac{y^2-1}{4}.
\end{equation} Thus,
\[
d=|T_n-2T_k|
=\left|\frac{x^2-1}{8}-\frac{y^2-1}{4}\right|
=\frac{|x^2-2y^2+1|}{8},
\] from which \eqref{e3376} follows.

Conversely, assume that $x,y\in\NN$ are odd and satisfy \eqref{e3376}. Set $n=\frac{x-1}{2}$ and $k=\frac{y-1}{2}$. Then $k,n\in\NN_0$ and \eqref{e9176} holds. Thus, 
\[T_n-2T_k=\frac{x^2-1}{8}-\frac{y^2-1}{4}
=\frac{x^2-2y^2+1}{8}\in\{d,-d\}.\] Thus, $|T_n-2T_k|=d$, and the proof is complete.
\end{proof}

\begin{corollary}
Let $d\in\NN_0$ and assume that there exist primes $p,q$, each
congruent to $3$ or $5$ modulo $8$, such that
$v_p(8d-1)$ and $v_q(8d+1)$ are both odd. Then $d$ does not appear as a value of \seqnum{A181176}.
\end{corollary}

\begin{proof}
By a comment to \seqnum{A035251}, a number $n\in\NN$ is representable in the form $x^2 - 2y^2$, for some $x,y\in\NN_0$ if and only if for every prime factor $r$ of $n$, which is congruent to $3$ or $5$ modulo $8$, the exponent $v_r(n)$ is even. Moreover,
\[
(x+2y)^2-2(x+y)^2=-(x^2-2y^2).
\]
Thus, an integer is representable in the form $x^2-2y^2$ if and only if its negative is representable in this form. Hence, the same criterion applies to negative integers as well. 

Now, by Lemma~\ref{l176}, for $d$ to be representable as $|T_n-2T_k|$ it is necessary that either $8d-1$ or $-8d-1$ be representable as $x^2-2y^2$. By the assumptions and the comment, neither $8d-1$ nor $8d+1$ is representable as
$x^2-2y^2$. Since $-8d-1=-(8d+1)$, this holds for $-8d-1$ as well. Thus, neither $8d-1$ nor $-8d-1$ is representable
as $x^2-2y^2$. Hence, $d$ is not representable as $|T_n-2T_k|$ and
therefore cannot occur as a value of \seqnum{A181176}.
\end{proof}

\begin{example}
In the following table we verify that each number $d$ conjectured in \seqnum{A181176} to be absent from the sequence is indeed not a value of the sequence. In the prime factorization of $8d-1$ and $8d+1$ a prime factor which is congruent to $3$ or $5$ modulo $8$ and has odd exponent is written in large font.
\[
\renewcommand{\arraystretch}{1.15}
\begin{tabular}{c|l|l}
$d$ & $8d-1$ & $8d+1$\\
\hline
$7$   & $55=\bigprime{5}\cdot 11$
      & $57=\bigprime{3}\cdot 19$\\
$18$  & $143=\bigprime{11}\cdot 13$
      & $145=\bigprime{5}\cdot 29$\\
$23$  & $183=\bigprime{3}\cdot 61$
      & $185=\bigprime{5}\cdot 37$\\
$31$  & $247=\bigprime{13}\cdot 19$
      & $249=\bigprime{3}\cdot 83$\\
$37$  & $295=\bigprime{5}\cdot 59$
      & $297=\bigprime{3}^{3}\cdot 11$\\
$38$  & $303=\bigprime{3}\cdot 101$
      & $305=\bigprime{5}\cdot 61$\\
$40$  & $319=\bigprime{11}\cdot 29$
      & $321=\bigprime{3}\cdot 107$\\
$47$  & $375=\bigprime{3}\cdot 5^3$
      & $377=\bigprime{13}\cdot 29$\\
$52$  & $415=\bigprime{5}\cdot 83$
      & $417=\bigprime{3}\cdot 139$\\
$59$  & $471=\bigprime{3}\cdot 157$
      & $473=\bigprime{11}\cdot 43$\\
$67$  & $535=\bigprime{5}\cdot 107$
      & $537=\bigprime{3}\cdot 179$\\
$68$  & $543=\bigprime{3}\cdot 181$
      & $545=\bigprime{5}\cdot 109$\\
$70$  & $559=\bigprime{13}\cdot 43$
      & $561=\bigprime{3}\cdot 11\cdot 17$\\
$73$  & $583=\bigprime{11}\cdot 53$
      & $585=3^2\cdot \bigprime{5}\cdot 13$\\
$83$  & $663=\bigprime{3}\cdot 13\cdot 17$
      & $665=\bigprime{5}\cdot 7\cdot 19$\\
$86$  & $687=\bigprime{3}\cdot 229$
      & $689=\bigprime{13}\cdot 53$\\
$88$  & $703=19\cdot \bigprime{37}$
      & $705=\bigprime{3}\cdot 5\cdot 47$\\
$92$  & $735=\bigprime{3}\cdot 5\cdot 7^2$
      & $737=\bigprime{11}\cdot 67$\\
$98$  & $783=\bigprime{3}^{3}\cdot 29$
      & $785=\bigprime{5}\cdot 157$\\
$102$ & $815=\bigprime{5}\cdot 163$
      & $817=\bigprime{19}\cdot 43$
\end{tabular}
\]
\end{example}

\section{Generating functions}

\subsection{\texorpdfstring{\seqnum{A000122}}{}}

Let $A(x)=\theta_3(0,x)=1+2\sum_{n=1}^\infty x^{n^2}$, where $\theta_3(z,x)$ is the Jacobi Theta function (e.g., \cite[(10.7.3)]{AAR}). Let $\sigma(n)$ denote the sum of divisors function, i.e., if $n\in\NN$ then \[\sigma(n)=\sum_{d\mid n}d.\] It is well-known that $\sigma$ is multiplicative, i.e., $\sigma(mn)=\sigma(m)\sigma(n)$ for every $m,n\in\NN$ with $\gcd(m,n)=1$ (e.g. \cite[Theorem 6.3]{B}). Furthermore, by \cite[Theorem 6.2(b)]{B}, if $n=p_1^{k_1}p_2^{k_2}\cdots p_r^{k_r}$ is the prime factorization of $n\in\NN$, then
\begin{equation}\label{factorrr}
\sigma(n)=
\frac{p_1^{k_1+1}-1}{p_1-1}
\cdot
\frac{p_2^{k_2+1}-1}{p_2-1}
\cdots
\frac{p_r^{k_r+1}-1}{p_r-1}.
\end{equation} 
The statement of the following theorem was conjectured by Bala in \seqnum{A000122}. 
\begin{theorem}
\begin{equation}\label{ep1}
A(x)
=\exp\left(2\sum_{n=0}^\infty \frac{x^{2n+1}}{(2n+1)(1+x^{2n+1})}\right).
\end{equation}
\end{theorem}

\begin{proof}
By the triple product identity (e.g., \cite[(10.4.1)]{AAR}),
\[
A(x)
=\prod_{n=1}^\infty (1-x^{2n})(1+x^{2n-1})^2.
\]
Taking logarithms on both sides yields
\begin{equation}\label{e1A122}
\log A(x)
=\sum_{n=1}^\infty\log(1-x^{2n})
+2\sum_{n=1}^\infty\log(1+x^{2n-1}).
\end{equation}
Using the standard power series
\[
\log(1-t)=-\sum_{k=1}^\infty\frac{t^k}{k},\qquad
\log(1+t)=\sum_{k=1}^\infty\frac{(-1)^{k+1}}{k}t^k,
\]
we obtain from \eqref{e1A122}
\[
\log A(x)
=-\sum_{n=1}^\infty \sum_{k=1}^\infty\frac{x^{2nk}}{k}
+2\sum_{n=1}^\infty\sum_{k=1}^\infty\frac{(-1)^{k+1}}{k}x^{(2n-1)k}.
\]
Write $\log A(x) = \sum_{m=1}^\infty (c^{(1)}_m+c^{(2)}_m) x^m$, where $c^{(1)}_m$ and $c^{(2)}_m$ are the coefficients of $x^m$ in the first and second double sum, respectively.
Let $m\in\NN$. If $m$ is odd, then $c^{(1)}_m=0$. Thus, assume that $m$ is even. Then
\[
c^{(1)}_m
=-\sum_{\substack{n,k\in\NN\\2nk=m}}\frac{1}{k}=-\sum_{\substack{n\in\NN\\n\mid \frac{m}{2}}}\frac{1}{\frac{m}{2n}}
   =-\frac{1}{m}\sum_{\substack{n\in\NN\\n\mid \frac{m}{2}}}2n
   =-\frac{1}{m}\sum_{\substack{d\mid m\\ d\textnormal{ even}}}d.
\] 
Now, let $m\in\NN$. We have\[
c^{(2)}_m
=2\sum_{\substack{n,k\in\NN\\ (2n-1)k=m}}\frac{(-1)^{k+1}}{k}=2\sum_{\substack{d\mid m\\ d \text{ odd}}}
   \frac{(-1)^{\frac{m}{d}+1}}{\frac{m}{d}}
=\frac{2}{m}\sum_{\substack{d\mid m\\ d \text{ odd}}}
   (-1)^{\frac{m}{d}+1}d.
\] It follows that 
\begin{equation}\label{eq:CNA122}
\log A(x)=-\sum_{m=1}^\infty
\frac{1}{m}\left(\sum_{\substack{d\mid m\\ d\textnormal{ even}}}d+2\sum_{\substack{d\mid m\\ d \text{ odd}}}
   (-1)^{\frac{m}{d}}d\right)x^m.
\end{equation}

Now, let $B(x)$ denote the function on the right-hand side of \eqref{ep1}. For $n\in\NN$ we have the expansion $\frac{1}{1+x^n}=\sum_{k=0}^\infty (-1)^k x^{nk}$. Thus,
\[
\log B(x)
=2\sum_{\substack{n\in\NN\\ n\text{ odd}}}
\frac{x^n}{n(1+x^n)}=2\sum_{\substack{n\in\NN\\ n\text{ odd}}}
\sum_{k=1}^\infty\frac{(-1)^{k-1}}{n}x^{nk} = -\sum_{m=1}^\infty \left(2\sum_{\substack{d\mid m\\ d\text{ odd}}}
  \frac{(-1)^{\frac{m}{d}}}{d}\right) x^m.
\] It remains to show that, for every $m\in\NN$, we have 
\[\frac{1}{m}\left(\sum_{\substack{d\mid m\\ d\textnormal{ even}}}d+2\sum_{\substack{d\mid m\\ d \text{ odd}}}
   (-1)^{\frac{m}{d}}d\right)=2\sum_{\substack{d\mid m\\ d\text{ odd}}}
  \frac{(-1)^{\frac{m}{d}}}{d},\] or, equivalently, that 
\begin{equation}\label{z1A122}
\sum_{\substack{d\mid m\\ d\textnormal{ even}}} d
=
2\sum_{\substack{d\mid m\\ d\textnormal{ odd}}}
(-1)^{\frac{m}{d}}\left(\frac{m}{d}-d\right).
\end{equation}
Let $L_m$ and $R_m$ denote the left-hand side and right-hand side of \eqref{z1A122}, respectively. First assume that $m$ is odd. Then $L_m=0$. On the other hand, all divisors of $m$ are odd. Hence, we need to show that 
\[
\sum_{d\mid m}
(-1)^{\frac{m}{d}}\left(\frac{m}{d}-d\right)=0.
\] Let $d$ be a divisor of $m$. Since $m$ is odd, $\frac{m}{d}$ is odd as well. Thus, $(-1)^{\frac{m}{d}}=-1$. Furthermore, the map $x\mapsto \frac{m}{x}$ is a bijection on the set of divisors of $m$. It follows that
\[
\sum_{d\mid m}
(-1)^{\frac{m}{d}}\left(\frac{m}{d}-d\right)
= \sum_{d\mid m} d-\sum_{d\mid m} \frac{m}{d}=0.
\] This concludes the proof that $R_m=L_m$, for odd $m$. Assume now that $m$ is even. Then, $m = 2^kn$, for some $k,n\in\NN$, with $n$ odd. Since $m$ is even, the even divisors of $m$ are precisely the numbers $2d$, where $d\mid \frac{m}{2}$. Hence,
\[
L_m=\sum_{d \mid \frac{m}{2}}2d=2\sigma\left(\frac{m}{2}\right)=2\sigma(2^{k-1}n). 
\] Since $\gcd(2^{k-1}, n)=1$, by the multiplicativity of $\sigma$, we have $\sigma(2^{k-1}n)=\sigma(2^{k-1})\sigma(n)$. By \eqref{factorrr}, $\sigma(2^{k-1})=2^k-1$. Thus, $L_m=2(2^{k}-1)\sigma(n)$. Now consider $R_m$. The odd divisors of $m$ are exactly the divisors of $n$. Furthermore, $\frac{m}{d}$ is even, for every odd divisor $d$ of $m$. Thus, $(-1)^{\frac{m}{d}}=1$. It follows that 
\begin{align*}
R_m&=2\sum_{d\mid n}
\left(\frac{m}{d}-d\right)\\
&=2\left(\frac{m}{n}\sum_{d\mid n}\frac{n}{d}-\sigma(n)\right)\\
&=2\left(\frac{m}{n}-1\right)\sigma(n)\\
&=2(2^k-1)\sigma(n).
\end{align*}
Thus, $R_m=L_m$ for even $m$ and the proof is complete.
\end{proof}

\subsection{\texorpdfstring{\seqnum{A026725}, \seqnum{A236830}, \seqnum{A108080}, and \seqnum{A026847}}{}}

Let $T(n,k)$ be the triangular array defined as follows. For $n\in\NN_0$ set $T(n,0)=T(n,n)=1$ and, for $n\in\NN$ and $k\in[n]$, let
\[
T(n+1,k)=
\begin{cases}
T(n,k-1)+T(n-1,k-1)+T(n,k),
& \text{if $k=\frac{n}{2}$,}\\
T(n,k-1)+T(n,k),
& \text{otherwise.}
\end{cases}
\] Let $C(x)$ denote the Catalan generating function. As is well-known, 
\begin{equation}\label{catalani}
C(x)=\frac{1-\sqrt{1-4x}}{2x}, \qquad C(x)=1+xC(x)^2.
\end{equation} The statement of Theorem \ref{A108080} was conjectured by Stephan in \seqnum{A108080}. We shall need the following result, which establishes a connection between the array $T(n,k)$ and the Riordan array $R(n,k)$ defined by
\begin{equation}\label{bb180}
R(n,k)=[x^n]\frac{(xC(x))^k}{1-xC(x)^3},\qquad n,k\in\NN_0.
\end{equation}

\begin{proposition}
Let $r,n\in\NN_0$ with $n\geq r$. Then
\[
T(2n,n+r)=R(n+r+1,2r+1).
\]
\end{proposition}

\begin{proof}
By a formula in \seqnum{A026674}, the generating function $A(x)$ of $(T(2m-1,m-1))_{m\in\NN}$ is given by
\begin{equation}\label{ek180}
A(x)=\frac{xC(x)^3}{1-xC(x)^3}.
\end{equation} For $n\in\NN_0$ set
$P_n(y)=\sum_{k=0}^n T(n,k)y^k$. We claim that for every $n\in\NN$ we have
\begin{align}
P_n(y)&=(1+y)P_{n-1}(y)\nonumber\\
&+
\begin{cases}
T(2m-1,m-1) y^m, & \text{if } n=2m+1 \text{ for some } m\in\NN,\\
0, & \text{otherwise.}
\end{cases} \label{e180}
\end{align} Indeed,
\[
(1+y)P_{n-1}(y)
=\sum_{k=0}^{n-1}T(n-1,k)y^k+\sum_{k=0}^{n-1}T(n-1,k)y^{k+1}.
\]
Thus, for every $k\in\NN_0$ with $k\leq n$,
\[
[y^k](1+y)P_{n-1}(y)=
\begin{cases}
T(n-1,0), & \textnormal{if } k=0,\\
T(n-1,k)+T(n-1,k-1), & \textnormal{if } k\in[n-1],\\
T(n-1,n-1), & \textnormal{if } k=n.
\end{cases}
\]
Since $T(n,0)=T(n,n)=T(n-1,0)=T(n-1,n-1)=1$, the coefficients of $y^0$ and $y^n$ in $P_n(y)$ and in $(1+y)P_{n-1}(y)$ agree. Let $k\in[n-1]$ such that $k\neq \frac{n-1}{2}$. By definition of $T(n,k)$, 
\[
T(n,k)=T(n-1,k-1)+T(n-1,k).
\] Thus, $[y^k]P_n(y)=[y^k](1+y)P_{n-1}(y)$. Now let $k=\frac{n-1}{2}$. Then $n=2k+1$ and by definition of $T(n,k)$,
\[
T(n,k)=T(n-1,k-1)+T(n-1,k)+T(2k-1,k-1).
\] Thus, $[y^k]P_n(y)=[y^k]((1+y)P_{n-1}(y)+T(2k-1,k-1)y^k)$. This concludes the proof of \eqref{e180}.

Now let $F(z,y)=\sum_{n=0}^\infty P_n(y)z^n$ be the generating function of $(P_n(y))_{n\in\NN_0}$. From \eqref{e180} it follows that
\[
F(z,y)=1+z(1+y)F(z,y)+zA(yz^2).
\]
Thus,
\begin{equation}\label{ez180}
F(z,y)=\frac{1+zA(yz^2)}{1-z(1+y)}.
\end{equation} Set
\[
E(t,y)=\sum_{n=0}^\infty P_{2n}(y)t^n.
\]
Using \eqref{ez180}, we have 
\begin{align*}
E(z^2,y)&=\sum_{n=0}^\infty P_{2n}(y)z^{2n}\\
&=\frac{F(z,y)+F(-z,y)}{2}\\
&=\frac12\left(\frac{1+zA(yz^2)}{1-z(1+y)}+\frac{1-zA(yz^2)}{1+z(1+y)}\right)\\
&=\frac{1+z^2(1+y)A(yz^2)}{1-z^2(1+y)^2}.
\end{align*} Substituting $z^2$ by $t$ and using \eqref{ek180} we obtain
\begin{equation}\label{el180}
E(t,y)=\frac{1+t(1+y)A(yt)}{1-t(1+y)^2}=\frac{1+t(1+y)\dfrac{ytC(yt)^3}{1-ytC(yt)^3}}{1-t(1+y)^2}.
\end{equation}
It follows that
\begin{align}
E(\tfrac{x}{y},y)&=\sum_{n=0}^\infty P_{2n}(y)x^ny^{-n}\nonumber\\
&=\sum_{n=0}^\infty\sum_{k=0}^{2n} T(2n,k)x^n y^{k-n}\nonumber\\ &=\sum_{r=-\infty}^\infty\left(\sum_{n=|r|}^\infty T(2n,n+r)x^n\right)y^r.\label{err80}
\end{align}
On the other hand, by \eqref{el180}, 
\begin{align*}
E(\tfrac{x}{y},y)&=\frac{1+\dfrac{x(1+y)}{y}\dfrac{xC(x)^3}{1-xC(x)^3}}{1-\dfrac{x}{y}(1+y)^2}\\
&=\frac{y(1-xC(x)^3)+x^2(1+y)C(x)^3}
{(1-xC(x)^3)(y-x(1+y)^2)}\\
&=\frac{C(x)}{1-xC(x)^{3}}\left(\frac{1}{1-xC(x)^{2}y}+\frac{xC(x)^{3}}{y-xC(x)^{2}}\right)\\
&=\frac{C(x)}{1-xC(x)^{3}}\left(\sum_{r=0}^\infty(xC(x)^2)^r y^r +xC(x)^{3}\sum_{r=0}^\infty(xC(x)^2)^r y^{-r-1}\right).
\end{align*}
From this and from \eqref{err80} it follows that for every $k\in\NN_0$,
\begin{equation}\label{gl180}
\sum_{n=k}^\infty T(2n,n+k)x^n =[y^k]E(\tfrac{x}{y},y)=
\frac{x^k C(x)^{2k+1}}{1-xC(x)^3}.
\end{equation} From this and from \eqref{bb180} it follows that for every $r,n\in\NN_0$ with $n\geq r$ we have
\begin{align*}
R(n+r+1,2r+1)&=[x^{n+r+1}]\frac{(xC(x))^{2r+1}}{1-xC(x)^3}\\
&=[x^n]\frac{x^rC(x)^{2r+1}}{1-xC(x)^3}\\
&=T(2n,n+r).\qedhere
\end{align*}
\end{proof}

\begin{theorem}\label{A108080}
Let $n\in\NN_0$. Then
\[\sum_{k=0}^n T(2n,n+k)=\sum_{i=0}^n\binom{2n+i}{n-i}.\]
\end{theorem}

\begin{proof}
Interchanging sums, we have
\[
\sum_{n=0}^\infty\sum_{k=0}^n T(2n,n+k)x^n=\sum_{k=0}^\infty\sum_{n=k}^\infty T(2n,n+k)x^n.
\]
With \eqref{gl180} we obtain
\begin{align*}
\sum_{k=0}^\infty\sum_{n=k}^\infty T(2n,n+k)x^n&=\sum_{k= 0}^\infty\frac{x^k C(x)^{2k+1}}{1-xC(x)^3}\\
&=\frac{C(x)}{1-xC(x)^3}\sum_{k=0}^\infty(xC(x)^2)^k\\
&=\frac{C(x)}{(1-xC(x)^3)(1-xC(x)^2)}\\
&=\frac{1}{\sqrt{1-4x}(1-xC(x)^3)},
\end{align*} where in the last transition we used \eqref{catalani}.
By a formula in \seqnum{A108080}, this is exactly the generating function of the numbers $\sum_{i=0}^n\binom{2n+i}{n-i}$, and the proof is complete.
\end{proof}

\subsection{\texorpdfstring{\seqnum{A158110}}{}}

Let $q,m\in\NN_0$ and define
\[
F_{q,m}(x)=
\exp\left(\sum_{n=1}^\infty\frac{q^{n^m}}{n}x^n\right).
\] The statement of the following theorem was conjectured by Hanna in \seqnum{A158110}. The case $m=2$ was proved by Bala in \seqnum{A155200}. Recall (e.g., \cite[Definition 6.3]{B}) that the M\"obius function $\mu$ is defined as follows. For $n\in\NN$ let
\[
\mu(n)=
\begin{cases}
1, & \text{if } n=1,\\
0, & \text{if } p^2\mid n \text{ for some prime } p,\\
(-1)^r, & \text{if } n=p_1p_2\cdots p_r,
\text{ where } p_i \text{ are distinct primes.}
\end{cases}
\] 

\begin{theorem}
The formal power series of the function $F_{q,m}(x)$ has integer coefficients.
\end{theorem}

\begin{proof}
If $q=0$, then $F_{q,m}(x)=1$ and the statement holds trivially. Thus, we assume that $q\ge 1$. We begin by proving the statement for $m=0$. We have
\[
\sum_{n=1}^\infty \frac{q^{n^0}}{n}x^n=q\sum_{n=1}^\infty \frac{x^n}{n}=-q\log(1-x)=\log(1-x)^{-q}.
\]
Hence,
\[
F_{q,0}(x)
=
\exp\left(\sum_{n=1}^\infty \frac{q^{n^0}}{n}x^n\right)=(1-x)^{-q}=\sum_{n=0}^\infty \binom{-q}{n}(-x)^n
=
\sum_{n=0}^\infty \binom{q+n-1}{n}x^n.
\] Since the binomial coefficients $\binom{q+n-1}{n}$ are integers, the formal power series $F_{q,0}(x)$ has integer coefficients.

Assume now that $m\geq 1$. By \cite[Exercise 5.2.a.]{St}, it suffices to show that $\sum_{d\mid n}\mu(d)q^{(\frac{n}{d})^m}\equiv 0 \pmod n$, for every $n\in\NN$. For $n=1$ the congruence holds trivially. Thus, assume that $n\geq 2$ and let $p$ be a prime divisor of $n$ and write $n=p^kN$ with $k\in\NN$ and $p\nmid N$. Every divisor $d$ of $n$ is therefore of the form $d = p^j d'$ with $0\leq j\leq k$ and $d'\mid N$. For $d'\mid N$ set $E_{d'}=(\tfrac{p^{k-1} N}{d'})^m$ and 
$\Delta_{d'}=
(\tfrac{N}{d'})^m p^{(k-1)m}(p^m-1)$. Since $\mu$ is multiplicative (e.g., \cite[Theorem 6.5]{B}) and since $\mu(p^j)=0$ for $j\geq 2$, we have
\begin{align*}
\sum_{d\mid n}\mu(d)q^{(\frac{n}{d})^m}
&=\sum_{d'\mid N}\mu(d')q^{(\frac{n}{d'})^m}
+\sum_{d'\mid N}\mu(p d')q^{(\frac{n}{p d'})^m}\\
&=\sum_{d'\mid N}\mu(d')\left(q^{( \frac{p^kN}{d'})^m}- q^{( \frac{p^{k-1}N}{d'})^m}
\right)\\
&=\sum_{d'\mid N}\mu(d')q^{E_{d'}}(q^{\Delta_{d'}}-1).
\end{align*} Thus, it suffices to show that $p^k \mid q^{E_{d'}}(q^{\Delta_{d'}}-1)$. We distinguish between two cases.
\begin{enumerate}
\item $p\mid q$. Then
$v_p(q^{E_{d'}}) = E_{d'}\cdot v_p(q) \geq p^{k-1}\cdot 1\geq k$. Thus, $p^k \mid q^{E_{d'}}$ and therefore $p^k \mid q^{E_{d'}}(q^{\Delta_{d'}}-1)$. 
\item $p\nmid q$. By Euler's theorem (e.g., \cite[Theorem 7.5]{B}),
$q^{\varphi(p^k)}\equiv 1 \pmod{p^k}$.
We have $(p-1)\mid (p^m-1)$ and, since $m\geq 1$, $p^{k-1}\mid p^{(k-1)m}$. It follows that
$\varphi(p^k)=p^{k-1}(p-1)\mid \Delta_{d'}$. Hence, $q^{\Delta_{d'}}\equiv 1 \pmod{p^k}$, i.e., $p^k\mid (q^{\Delta_{d'}}-1)$. Thus, $p^k\mid q^{E_{d'}}(q^{\Delta_{d'}}-1)$.\qedhere
\end{enumerate}
\end{proof}

\subsection{\texorpdfstring{\seqnum{A239333}}{}}

For $n\in\NN$ let $a_n$ denote the number of words $w_1\cdots w_n\in\{0,1,2,3\}^n$ of length $n$ such that for every $i\in[n]$,
\begin{equation}\label{mod4}
w_i\not\equiv 1 \pmod 4,
\qquad w_i\not\equiv 3+\sum_{j=1}^{i-1}w_j \pmod 4.
\end{equation} The statement of the following theorem was conjectured by Barker in \seqnum{A239333}. 

\begin{theorem}
Let $n\in\NN$. Then $a_{n+3} = 2a_{n+2}+2a_{n}$.
\end{theorem}

\begin{proof}
For $r\in\{0,1,2,3\}$ let $v_n(r)$ denote the number of words $w_1\cdots w_n\in\{0,1,2,3\}^n$ of length $n$ satisfying \eqref{mod4} such that $\sum_{j=1}^nw_j\equiv r\pmod 4$. Let $V_r(x)$ be the corresponding generating function, i.e., $V_r(x)=\sum_{n=1}^\infty v_n(r)x^n$. We obtain the recursions
\begin{align*}
v_{n+1}(0)&=v_n(0)+v_n(1)+v_n(2),\\
v_{n+1}(1)&=v_n(2),\\
v_{n+1}(2)&=v_n(0)+v_n(2)+v_n(3),\\
v_{n+1}(3)&=v_n(1)+v_n(3).
\end{align*} With the initial values 
$v_1(0)=v_1(2)=1,v_1(1)=v_1(3)=0$, the recursions may be written in terms of the generating functions as follows:
\begin{align*}
V_0(x) &= x(1+V_0(x)+V_1(x)+V_2(x)),\\
V_1(x) &= xV_2(x),\\
V_2(x) &= x(1+V_0(x)+V_2(x)+V_3(x)),\\
V_3(x) &= x(V_1(x)+V_3(x)).
\end{align*}
Solving the system gives
\begin{align*}
V_0(x)&=\frac{x(1-x+x^2-2x^3)}{(1-x)(1-2x-2x^3)},\\
V_1(x)&=\frac{x^2}{1-2x-2x^3},\\
V_2(x)&=\frac{x}{1-2x-2x^3},\\
V_3(x)&=\frac{x^3}{(1-x)(1-2x-2x^3)}.
\end{align*}
Let $A(x)=\sum_{n=1}^\infty a_n x^n$ be the generating function of $(a_n)_{n\in\NN}$. Since $a_n = \sum_{r=0}^3v_n(r)$ for every $n\in\NN$, we have 
\[A(x)=\sum_{r=0}^3 V_r(x)
=\frac{x(2+x+2x^2)}{1-2x-2x^3},
\] from which the asserted recursion immediately follows.
\end{proof}

\subsection{\texorpdfstring{\seqnum{A255992}}{}}

For $k\in\NN$ and $m\in\NN_0$, let $b_k(m)$ denote the number of binary words of length $m$ such that in every block of $k$ consecutive neighbor pairs there is at most one downstep, i.e., at most one occurrence of the pair $10$. We obtain the corresponding generating function and derive from it an explicit formula for the numbers $b_k(m)$. These formulas confirm the conjectures made by Barker in \seqnum{A255992}.

\begin{theorem}
Let $k,m\in\NN$ such that $m\geq k+1$. Then $b_k(m)=[x^m]F_k(x)$, where
\[
F_k(x)=
\frac{1-2x+2x^2-(k-1)x^k+(k-2)x^{k+1}}
{(1-x)^2(1-2x+x^2-(k-1)x^k+(k-2)x^{k+1})}.\]
\end{theorem}

\begin{proof}
Let $r\in\NN_0$ and let $b_{k,r}(m)$ denote the number of binary words of length $m$ satisfying the condition above and having exactly $r$ downsteps. Let $w$ be such a word and assume that $r=0$. We claim that $b_{k,0}(m)=[x^m]B_{k,0}(x)$, where $B_{k,0}(x)=\frac{1}{(1-x)^2}$. Indeed, $w$ factors as $w=0^{\alpha_0}1^{\alpha_1}$, for some $\alpha_0,\alpha_1\in\NN_0$ and therefore
\begin{equation}\label{ed5992}
B_{k,0}(x)=\sum_{\alpha_0=0}^\infty\sum_{\alpha_1=0}^\infty x^{\alpha_0+\alpha_1}=\frac{1}{(1-x)^2}.
\end{equation}
Assume now that $r\geq 1$. Then $w$ factors as
\begin{equation}\label{e15992}
w=
0^{\alpha_0}1^{\alpha_1}0^{\alpha_2}1^{\alpha_3} \cdots
1^{\alpha_{2r-1}}0^{\alpha_{2r}}1^{\alpha_{2r+1}},
\end{equation}
where $\alpha_0,\alpha_{2r+1}\in\NN_0$ and 
$\alpha_i\in\NN$, for every $i\in[2r]$. For $j\in[r]$ let $p_j$ be the index of the left digit of the $j$th occurrence of $10$. Then, 
\[
p_{j+1}-p_j=\alpha_{2j}+\alpha_{2j+1},\qquad j\in[r-1].
\]
Since $m\ge k+1$, the condition on $w$ is therefore equivalent to
\begin{equation}\label{e25992}
\alpha_{2j}+\alpha_{2j+1}\geq k, \qquad j\in[r-1].
\end{equation}
We claim that $b_{k,r}(m)=[x^m]B_{k,r}(x)$, where \[B_{k,r}(x)=
\frac{x^2}{(1-x)^4}\left(
\frac{x^k((k-1)-(k-2)x)}{(1-x)^2}
\right)^{r-1}.
\] Indeed, in \eqref{e15992}, both blocks $0^{\alpha_0}$ and $1^{\alpha_{2r+1}}$ contribute a factor $\frac{1}{1-x}$ each. Similarly, the blocks $1^{\alpha_1}$ and $0^{\alpha_{2r}}$ contribute a factor $\frac{x}{1-x}$ each. By \eqref{e25992}, for every $j\in[r-1]$, the block
$0^{\alpha_{2j}}1^{\alpha_{2j+1}}$ contributes a factor
\[
P_k(x)=
\sum_{\substack{u,v\in\NN\\u+v\geq k}} x^{u+v}=\sum_{s\geq \max\{2,k\}} (s-1)x^s=\frac{x^k((k-1)-(k-2)x)}{(1-x)^2}.
\]
It follows that 
\begin{equation}\label{edd5992}
B_{k,r}(x)=\frac{x^2}{(1-x)^4}P_k(x)^{r-1} =\frac{x^2}{(1-x)^4}\left(\frac{x^k((k-1)-(k-2)x)}{(1-x)^2}\right)^{r-1}.
\end{equation}
Hence,
\begin{align*}
F_k(x)&=
\sum_{r=0}^\infty 
B_{k,r}(x)\\
&=\frac{1}{(1-x)^2}+\frac{x^2}{(1-x)^4}\sum_{r=1}^\infty\left(\frac{x^k((k-1)-(k-2)x)}{(1-x)^2}\right)^{r-1}\\
&=\frac{1}{(1-x)^2}+\frac{x^2}{(1-x)^4}\cdot\frac{1}{1-\dfrac{x^k((k-1)-(k-2)x)}{(1-x)^2}}\\
&=\frac{1-2x+2x^2-(k-1)x^k+(k-2)x^{k+1}}{(1-x)^2(1-2x+x^2-(k-1)x^k+(k-2)x^{k+1})}.
\end{align*}
Finally,
\[
b_k(m)=\sum_{r=0}^\infty b_{k,r}(m)=\sum_{r=0}^\infty [x^m]B_{k,r}(x)=[x^m]F_k(x). \qedhere
\]
\end{proof}

\begin{lemma}
Let $k,n\in\NN$. Then
\[b_k(n+k)
=
n+k+1+
\sum_{r=1}^{\left\lfloor\frac{n-2}{k}\right\rfloor+2}\sum_{j=0}^{r-1}
(-1)^j\binom{r-1}{j}(k-1)^{r-1-j}(k-2)^j
\binom{n+(2-r)k+2r-1-j}{2r+1}.
\] 
\end{lemma}

\begin{proof}
By \eqref{ed5992} and \eqref{edd5992}, and using 
$[x^m]\frac{1}{(1-x)^s}=\binom{m+s-1}{s-1}$, we have
\begin{align*}
b_k(n+k)&=[x^{n+k}]\frac{1}{(1-x)^2}
+
\sum_{r=1}^\infty
[x^{n+k}]\frac{x^2}{(1-x)^4}\left(
\frac{x^k((k-1)-(k-2)x)}{(1-x)^2}
\right)^{r-1}\\
&=n+k+1
+
\sum_{r=1}^\infty
[x^{n-2-(r-2)k}]
\frac{((k-1)-(k-2)x)^{r-1}}{(1-x)^{2r+2}}\\
&=n+k+1
+
\sum_{r= 1}^\infty \sum_{j=0}^{r-1}
(-1)^j\binom{r-1}{j}(k-1)^{r-1-j}(k-2)^j
[x^{n-2-(r-2)k-j}]
\frac{1}{(1-x)^{2r+2}}\\
&=n+k+1+
\sum_{r=1}^\infty\sum_{j=0}^{r-1}
(-1)^j\binom{r-1}{j}(k-1)^{r-1-j}(k-2)^j
\binom{n+(2-r)k+2r-1-j}{2r+1}.
\end{align*}
Since no word of length $n+k$ can have more than $\left\lfloor\frac{n-2}{k}\right\rfloor+2$ downsteps, the assertion follows.
\end{proof}

\subsection{\texorpdfstring{\seqnum{A258902}}{}}

Let $(a_n)_{n\in\NN}$ be the sequence defined by the exponential generating function
\[
A(x) = \sum_{n=1}^\infty a_n \frac{x^n}{n!},
\]
where $A(x)$ is the series reversion of $f(x) = x - \frac{x^2}{2} -
\frac{x^3}{3}$, in the sense that
$f(A(x)) = x$. The statement concerning the recurrence in Theorem \ref{kos} below was conjectured by Kotesovec in \seqnum{A258902}. We shall need the following result.

\begin{lemma}\label{l1902}
Let $n\in\NN$ and let $\alpha,\beta\in\RR$. Then, for every $k\in\NN_0$,
\begin{equation}\label{e8}
[x^k](1 - \alpha x - \beta x^2)^{-n}=\sum_{m=0}^{\left\lfloor \frac{k}{2}\right\rfloor}
\binom{n+k-m-1}{k-m}\binom{k-m}{m}
\alpha^{k-2m}\beta^{m}.
\end{equation}
\end{lemma}

\begin{proof}
By the binomial theorem for negative integer exponents (e.g., \cite[(5.56)]{CONC}), 
\begin{align*}
(1-\alpha x-\beta x^2)^{-n}
&=
\sum_{j=0}^\infty \binom{n+j-1}{j}(\alpha x+\beta x^2)^j\\
&=\sum_{j=0}^\infty \binom{n+j-1}{j}\sum_{i=0}^j \binom{j}{i} \alpha^{j-i}\beta^i x^{j+i}.
\end{align*} Now, a monomial $x^{j+i}$ contributes
to $[x^k]$ if and only if  $j+i=k$, i.e., exactly when $i=k-j$. Since $0\leq i\leq j$, necessarily $0 \leq k-j \leq j$, or, equivalently, $\frac{k}{2} \leq j \leq k$. It follows that
\[
[x^k](1-\alpha x-\beta x^2)^{-n}=
\sum_{j=\left\lceil \frac{k}{2}\right\rceil}^{k}
\binom{n+j-1}{j}\binom{j}{k-j}\alpha^{2j-k}\beta^{k-j}.
\]
Substituting $m=k-j$ immediately yields \eqref{e8}.
\end{proof}

\begin{theorem}\label{kos}
Let $n\in\NN$. Then
\begin{equation}\label{eekos}
a_n=
\frac{(n-1)!}{2^{n-1}}
\sum_{m=0}^{\left\lfloor \frac{n-1}{2}\right\rfloor}
\binom{2n-2-m}{n-1-m}\binom{n-1-m}{m}\left(\frac{4}{3}\right)^m.
\end{equation}
Furthermore, 
\begin{equation}\label{e10902}
19a_{n+2}
-(42n+21)a_{n+1}
+(4-36n^2)a_n=0.
\end{equation}
\end{theorem}

\begin{proof}
For $m\in\NN$ let $b_m =\frac{a_m}{m!}$. Then 
$A(x)=\sum_{m=1}^{\infty} b_m x^m$. Let $\phi(x) = (1 - \frac{x}{2} -\frac{x^2}{3})^{-1}$. We have 
\[x=f(A(x))=\frac{A(x)}{\phi(A(x))}.
\] By the Lagrange inversion theorem (e.g., \cite[p.\ 66]{FS}), for every $n\in\NN$,
\begin{equation}
a_n=n!b_n=(n-1)![x^{n-1}]\phi(x)^n.
\end{equation} Using Lemma \ref{l1902} with $\alpha = \frac{1}{2},\beta =
\frac{1}{3}$, and $k=n-1$ we obtain \eqref{eekos}.

To see that \eqref{e10902} holds, set $y=A(x)$. Then
\begin{equation}\label{inver1}
x=y-\frac{y^2}{2}-\frac{y^3}{3}.
\end{equation}
Differentiating we obtain
$1=(1-y-y^2)y'$. Hence,
\begin{equation}\label{eq:firstder258902}
y'=\frac{1}{1-y-y^2}.
\end{equation}
Differentiating again we obtain
\begin{equation}\label{eq:secondder258902}
y''=\frac{1+2y}{(1-y-y^2)^3}.
\end{equation}
Using \eqref{inver1}, a direct calculation gives
\begin{align*}
(1-y-y^2)^2(19-4y-4y^2)
&=19-42x-36x^2,\\
3(1+2y)(7-2y-2y^2)&=21+36x.
\end{align*}
It follows from \eqref{eq:firstder258902} and
\eqref{eq:secondder258902} that
\begin{align*}
&(19-42x-36x^2)y''-(21+36x)y'+4y+2\\
&\quad=
\frac{(1+2y)(19-4y-4y^2)}{1-y-y^2}-\frac{3(1+2y)(7-2y-2y^2)}{1-y-y^2}+4y+2\\
&\quad=
\frac{1+2y}{1-y-y^2}(19-4y-4y^2-3(7-2y-2y^2))+4y+2\\
&\quad=
\frac{2(1+2y)(y+y^2-1)}{1-y-y^2}+4y+2\\
&\quad=-2(1+2y)+4y+2\\
&\quad=0.
\end{align*}
Thus, $A(x)$ satisfies the differential equation
\begin{equation}\label{diffeq1}
(19-42x-36x^2)A''(x)
-(21+36x)A'(x)
+4A(x)+2=0.
\end{equation}
Using
$A(x)=\sum_{n=1}^{\infty}a_n\frac{x^n}{n!}$, by
equating the coefficients of $\frac{x^n}{n!}$ in
\eqref{diffeq1}, for $n\in\NN$, the assertion follows.
\end{proof}

\subsection{\texorpdfstring{\seqnum{A293129}}{}}
Let 
\[
L(x)= \sum_{m=-\infty}^{\infty}\frac{(x-x^{2m-1})^{2m-1}}{2m-1}.
\]
Let $(a_n)_{n\in\NN}$ be the sequence corresponding to the logarithmic generating function, i.e.,
\begin{equation}\label{ew1129}
L(x)=\sum_{n=1}^\infty a_n\frac{x^{2n-1}}{2n-1}.
\end{equation}
The statement of the following theorem was conjectured by Hanna in \seqnum{A293129}.

\begin{theorem}
Let $n\in\NN$. Then 
$a_{2^n+1}=1$.
\end{theorem}

\begin{proof}
By a formula in \seqnum{A293129}, 
\begin{equation}\label{ew11299}
L(x)= -\log(1-x)-\sum_{\substack{m\in\ZZ\\ m\neq 0}}\frac{(x-x^{2m})^{2m}}{2m}.
\end{equation}
Comparing \eqref{ew1129} with \eqref{ew11299}, we have $a_{2^{n}+1}=(2^{n+1}+1)[x^{2^{n+1}+1}]L(x)$. 
Since 
$-\log(1-x)=\sum_{m=1}^\infty\frac{x^m}{m}$,
we have
$[x^{2^{n+1}+1}](-\log(1-x))=\frac{1}{2^{n+1}+1}$. We claim that for every $0\neq m\in\ZZ$, we have
$[x^{2^{n+1}+1}](x-x^{2m})^{2m}=0$. To see this, we distinguish between two cases.
\begin{enumerate}
\item $m\geq 1$. We have
\[
(x-x^{2m})^{2m}=x^{2m}(1-x^{2m-1})^{2m}
=\sum_{k=0}^{2m}(-1)^k\binom{2m}{k}x^{2m+k(2m-1)}.
\]
Assume that for some $0\leq k\leq 2m$, we have $2^{n+1}+1=2m+k(2m-1)$. Then $2^{n+1}+1\equiv 1\pmod{2m-1}$. Thus, $(2m-1)\mid 2^{n+1}$. But $2m-1$ is odd while $2^{n+1}$ is a power of $2$. Thus, $m=1$ and therefore $2^{n+1}+1=2+k$. Hence, $k=2^{n+1}-1\geq 3$, contradicting $k\leq 2m=2$. 
\item $m\leq -1$. Set $u=-m$. Then
\[
(x-x^{2m})^{2m}=
x^{4u^2}(1-x^{2u+1})^{-2u}
=\sum_{k=0}^{\infty}\binom{2u+k-1}{k}x^{4u^2+k(2u+1)}.
\]
Assume that for some $k\in\NN_0$ we have $2^{n+1}+1=4u^2+k(2u+1)$. Since $4u^2+k(2u+1)=(2u-1)(2u+1)+1+k(2u+1)$, we have $2^{n+1}+1\equiv 1\pmod{2u+1}$. Thus, $(2u+1)\mid 2^{n+1}$. But $2u+1\geq 3$ is odd, while $2^{n+1}$ is a power of $2$, a contradiction. 
\end{enumerate} It follows that 
\begin{align*}
[x^{2^{n+1}+1}]L(x)&=[x^{2^{n+1}+1}](-\log(1-x))-[x^{2^{n+1}+1}]\sum_{\substack{m=-\infty\\ m\neq 0}}^{\infty}\frac{(x-x^{2m})^{2m}}{2m}\\
&=\frac{1}{2^{n+1}+1} + 0 = \frac{1}{2^{n+1}+1},
\end{align*} and the proof is complete.
\end{proof}

\subsection{\texorpdfstring{\seqnum{A325046}}{}}

Let 
\[
A(x)=\sum_{n=0}^\infty \frac{x^n(1+x^n)^n}{(1-x^{n+1})^{n+1}}.
\] Let $(a_n)_{n\in\NN_0}$ be the sequence whose generating function
$\sum_{n=0}^\infty a_nx^n$ 
corresponds to $A(x)$. 

We write $r\preceq s$ if every binary digit equal to $1$ in $r$ is also equal to $1$ in $s$. Equivalently, $r\land s=r$. We also write $r\perp s$ if $r$ and $s$ have no common positions of ones. Equivalently $r\land s=0$.

The statement of the following theorem was conjectured by Hanna in \seqnum{A325046}.

\begin{theorem}
Let $N\in\NN_0$. Then $a_N$ is odd if and only if $N=m(m+1)$, for some $m\in\NN_0$.
\end{theorem}

\begin{proof}
Let $n\in\NN_0$. We have
\[
\frac{1}{(1-x^{n+1})^{n+1}}=\sum_{k=0}^\infty\binom{n+k}{k}x^{(n+1)k},
\qquad
(1+x^n)^n=\sum_{i=0}^n \binom{n}{i}x^{ni}.
\] Thus, 
\[
A(x)=\sum_{n=0}^\infty \sum_{i=0}^n\sum_{k=0}^\infty
\binom{n}{i}\binom{n+k}{k}\,x^{\,n(i+1)+(n+1)k}.
\] Thus, the parity of $a_N$ is equal to the parity of $\# P_N$, where
\[
P_N=\left\{(n,i,k)\in\NN_0^3\;:\; i\leq n,\ N=n(i+1)+(n+1)k,\ 
\binom{n}{i}\binom{n+k}{k} \textnormal{ is odd}\right\}.
\] By Lucas' theorem modulo $2$ (e.g., \cite{L}), for every $n,i,k\in\NN_0$ we have
\begin{align*}
\binom{n}{i}\equiv 1 \pmod 2
\quad & \iff\quad
i\preceq n,\\
\binom{n+k}{k}\equiv 1 \pmod 2
\quad &\iff \quad
n\perp k.
\end{align*} 
Thus,
\[
P_N=
\{(n,i,k)\in\NN_0^3\;:\; N=n(i+1)+(n+1)k,\ i\preceq n,\ n\perp k\}.
\] Define
\[
Q_N=\left\{(n,u)\in\NN_0^2:\ N=nu+(n\lor u)\right\}.
\]
We claim that $\#P_N=\#Q_N$. To this end, define a map $\Psi\colon P_N\to Q_N$ by $\Psi(n,i,k)=(n,i+k)$. To see that $\Psi$ is well-defined, let $(n,i,k)\in P_N$. Since $i\preceq n$ and $n\perp k$, we have $i\perp k$ and therefore $i+k=i\lor k$.
Furthermore,
\[
n\lor(i\lor k)=(n\lor i)\lor k = n\lor k=n+k.
\]
Therefore
\[
N=n(i+1)+(n+1)k=n(i+k)+n+k=n(i+k)+(n\lor (i+k)).
\] Thus, $(n,i+k)\in Q_N$ and $\Psi$ is well-defined.

Conversely, define a map $\Phi \colon Q_N\to P_N$ by $\Phi(n,u)=(n,i,k)$, where $i=n\land u$ and $k=u-i$. To see that $\Phi$ is well-defined, let $(n,u)\in Q_N$. Then $i\preceq n, k\perp n$, and $u=i+k$. Furthermore,
$n\lor u=n+k$. Hence
\[
N=nu+(n\lor u)=n(i+k)+n+k=n(i+1)+(n+1)k.
\] Thus, $(n,i,k)\in P_N$ and $\Phi$ is well-defined.

It is immediate from the definitions that
$\Phi$ and $\Psi$ are inverses. It follows that $\#P_N=\#Q_N$.

It remains to determine the parity of $\# Q_N$. To this end notice that the condition in $Q_N$ is symmetric, i.e., $(n,u)\in Q_N \iff (u,n)\in Q_N$. Thus, elements $(n,u)$ in $Q_N$ with $n\neq u$ come in pairs and therefore do not affect the parity of $\# Q_N$. It remains to count elements in $Q_N$ of the form $(n,n)$. In this case, the condition in $Q_N$ has the form $N=n^2+(n\lor n)=n^2+n=n(n+1)$. Hence, $a_N$ is odd if and only if $N=n(n+1)$ for some $n\in\NN_0$.
\end{proof}

\subsection{\texorpdfstring{\seqnum{A355372}}{}}

Let 
\[
F(x)=\frac{\log\left(\frac{1-x}{1-2x}\right)}{(1-x)^3}.
\]
Let $(a_n)_{n\in\NN_0}$ be the sequence corresponding to the exponential generating function, i.e.,
\[
F(x)=\sum_{n=0}^\infty a_n\frac{x^n}{n!}.
\]
The statement of Theorem \ref{t172} was conjectured by Tebni in \seqnum{A355372}. We shall need the following result.

\begin{lemma}\label{l172}
Let $n\in\NN$. Then
\[
a_n = n!\sum_{m=1}^n \frac{1}{m}\binom{n+2}{m+2}.
\]
\end{lemma}

\begin{proof}
Using $-\log(1-x)= \sum_{m=1}^\infty \frac{x^m}{m}$ we have
\[
\log\left(\frac{1-x}{1-2x}\right) = -\log\left(1-\frac{x}{1-x}\right)
= \sum_{m=1}^\infty \frac{1}{m}
\left(\frac{x}{1-x}\right)^m.
\]
Hence,
\[
F(x)= \sum_{m=1}^\infty \frac{1}{m} x^m (1-x)^{-(m+3)}.
\] Using 
\[
(1-x)^{-r} = \sum_{k=0}^\infty \binom{r+k-1}{k} x^k,\qquad r\in\NN,
\] we obtain, for every $m\in\NN$,
\[
(1-x)^{-(m+3)}
= \sum_{k=0}^\infty \binom{m+2+k}{m+2} x^k.
\]
It follows that
\[
F(x)= \sum_{m=1}^\infty\sum_{k=0}^\infty
\frac{1}{m}\binom{m+2+k}{m+2}x^{m+k}.
\] For every $m\in\NN, k\in\NN_0$, we have $n=m+k\iff k=n-m$. Thus,
\[
[x^n]F(x)
= \sum_{m=1}^n \frac{1}{m}\binom{m+2+(n-m)}{m+2}
= \sum_{m=1}^n \frac{1}{m}\binom{n+2}{m+2}.
\] Since $a_n=n![x^n]F(x)$, the assertion follows.
\end{proof}

\begin{theorem}\label{t172}
Let $p$ be a prime number. Then
$a_p \equiv -1 \pmod p$.
\end{theorem}

\begin{proof}
By Lemma \ref{l172},
\[
a_p= p! \sum_{m=1}^p \frac{1}{m}\binom{p+2}{m+2} =  \sum_{m=1}^{p-1} \frac{p!}{m}\binom{p+2}{m+2}+ (p-1)!.
\] Let $m\in[p-1]$. Then $m\mid (p-1)!$ and therefore $\frac{p!}{m} = \frac{p(p-1)!}{m}$ is divisible by $p$. Thus, $p\mid \frac{p!}{m}\binom{p+2}{m+2}$ and therefore
\[
p! \sum_{m=1}^{p-1} \frac{1}{m}\binom{p+2}{m+2} \equiv 0 \pmod p.
\] Using Wilson's theorem, the proof is complete.
\end{proof}

\subsection{\texorpdfstring{\seqnum{A359087}}{}}

Let $n\in\NN$. We successively construct a triangular array as follows.
For $-n+1\leq j\leq n-1$, let $e(n-1,j)=n-|j|$. Now assume that the $k$th row has already been defined, where
$1\leq k\leq n-1$. Then, for $-k+1\leq j\leq k-1$, let
\[
e(k-1,j)=e(k,j-1)+e(k,j)+e(k,j+1).
\]
Set $a_n=e(0,0)$. The statement of the following theorem was conjectured in \seqnum{A359087} by Schott. It relates $(a_n)_{n\in\NN}$ with the sequence $(b_n)_{n\in\NN_0}$, defined as follows: For $n\in\NN_0$ let $b_n$ denote the number of $(1,0)$ steps in all paths of length $n$ with steps $(1,1), (1,-1)$, and $(1,0)$, starting at $(0,0)$, staying weakly above the $x$-axis.

\begin{theorem}
Let $n\in\NN$. Then
\begin{equation}\label{e1787}
a_n = n3^{n-1} - 2b_{n-1}.
\end{equation}
\end{theorem}

\begin{proof}
The key observation is that for every $-n+1\leq j\leq n-1$, the entry $e(n-1,j)$ contributes its value to the bottom element exactly once for each path from $(0,0)$ to $(n-1,j)$, where the path uses only the steps $(1,1)$, $(1,0)$, and $(1,-1)$. For $m\in\NN_0$ and $k\in\ZZ$, we denote by $T(m,k)$ the number of such paths of length $m$ ending at $(m,k)$. Thus, with the symmetry $T(n-1,-j)=T(n-1,j)$, we obtain
\begin{align*}
a_n& = \sum_{j=-n+1}^{n-1} e(n-1,j)T(n-1,j)\\
&= nT(n-1,0)
+ \sum_{j=1}^{n-1}(n-j)(T(n-1,j)+T(n-1,-j))\\
&= nT(n-1,0)
+ 2\sum_{j=1}^{n-1}(n-j)T(n-1,j)\\
&=n\left(T(n-1,0) + 2\sum_{j=1}^{n-1}T(n-1,j)\right)
   -2\sum_{j=1}^{n-1} jT(n-1,j)\\
&=n\sum_{j=-n+1}^{n-1}T(n-1,j)
   -2\sum_{j=1}^{n-1} jT(n-1,j).   
\end{align*}
Now, the expression \[\sum_{j=-n+1}^{n-1}T(n-1,j)\] merely counts all possible walks of length $n-1$, and their number is $3^{n-1}$. For $m\in\NN_0$, set
\[
S_m = \sum_{j=1}^{m} jT(m,j).
\] It remains to show that $S_m=b_m$. To this end, we first derive a recurrence for $S_m$. First, notice that $S_0=0$. Now, let $m\in\NN$ and $k\in\ZZ$. Using the recursion for $T(m,k)$,
\[
T(m,k)=T(m-1,k-1)+T(m-1,k)+T(m-1,k+1),
\]
we have
\begin{align}
S_m
&= \sum_{j=1}^{m} j T(m,j) \nonumber\\
&= \sum_{j=1}^{m} j (T(m-1,j-1)+T(m-1,j)+T(m-1,j+1) )\nonumber\\
&=\sum_{j=0}^{m-1} (j+1) T(m-1,j)+\sum_{j=1}^mjT(m-1,j)+\sum_{j=2}^{m+1}(j-1) T(m-1,j)\nonumber\\
&=\left(S_{m-1} + \sum_{j=0}^{m-1}T(m-1,j)\right)+S_{m-1}+\left(S_{m-1} - \sum_{j=0}^{m-1}T(m-1,j) + T(m-1,0)\right)\nonumber\\
&= 3S_{m-1} + T(m-1,0).\label{e9987}
\end{align}
Let $S(x)=\sum_{m=0}^\infty S_m x^m$ be the generating function for $(S_m)_{m\in\NN_0}$ and let $R(x)=\sum_{m=0}^\infty T(m,0)x^m$ be the generating function for $(T(m,0))_{m\in\NN_0}$. By a formula in \seqnum{A002426}, \[R(x)= \frac{1}{\sqrt{1-2x-3x^{2}}}.\] Multiplying \eqref{e9987} by $x^m$ and summing over $m\in\NN$ yields
\[
S(x) = 3x S(x) + x R(x).
\]
Solving for $S(x)$ we obtain
\[
S(x) = \frac{x}{(1-3x)\sqrt{1-2x-3x^2}}.
\]
By a formula in \seqnum{A132894}, this is precisely the generating function of $(b_n)_{n\in\NN_0}$. Thus, $S_m=b_m$, for every $m\in\NN_0$.
\end{proof}

\begin{corollary}
Let $A(x)$ be the generating function of $(a_n)_{n\in\NN}$. Then
\[
A(x)=\frac{x}{(1-3x)^2}
-\frac{2x^{2}}{(1+x)^{1/2}(1-3x)^{3/2}}.
\]
\end{corollary}

\section*{Declaration of generative AI and AI-assisted technologies
in the manuscript preparation process}

During the preparation of this manuscript, the author used ChatGPT
to improve its language, clarity, and presentation. The author
subsequently reviewed and edited the output and takes full
responsibility for the content of the manuscript.

\end{document}